\documentclass[pdflatex,sn-mathphys-ay]{sn-jnl}

\usepackage{amsmath, amsthm}

\usepackage{tikz}
\usetikzlibrary{positioning, arrows.meta, shapes}

\usepackage{amsmath, amssymb, amsthm} 
\usepackage{algorithm}
\usepackage{algpseudocode} 
\usepackage{bm} 
\usepackage{graphicx} 
\usepackage{hyperref} 

\newtheorem{corollary}{Corollary}

\usepackage{booktabs}

\usepackage{pgfplots}
\usepgfplotslibrary{fillbetween}

\usepackage{tabularx} 
\usepackage{booktabs} 

\theoremstyle{thmstyleone}%
\newtheorem{theorem}{Theorem}
\newtheorem{proposition}[theorem]{Proposition}%

\theoremstyle{thmstyletwo}%
\newtheorem{example}{Example}%
\newtheorem{remark}{Remark}%

\theoremstyle{thmstylethree}%
\newtheorem{definition}{Definition}%

\begin{document}

\title[MARL-CC]{MARL-CC: A Mathematical Framework for Multi-Agent Reinforcement Learning in Connected Autonomous Vehicles: Addressing Nonlinearity, Partial Observability, and Credit Assignment for Optimal Control}


\author*[1,2]{\fnm{Mazyar} \sur{Taghavi}}\email{mazyar\_taghavi@mathdep.iust.ac.ir}

\author*[1]{\fnm{Javad} \sur{Vahidi}}\email{jvahidi@iust.ac.ir}

\affil*[1]{\orgdiv{School of Mathematics and Computer Science}, \orgname{Iran University of Science and Technology}, \city{Tehran}, \country{Iran}}

\affil[2]{ \orgname{Intelligent Knowledge City}, \city{Isfahan}, \country{Iran}}


\abstract{Multi-Agent Reinforcement Learning (MARL) has emerged as a powerful paradigm for cooperative decision-making in connected autonomous vehicles (CAVs); however, existing approaches often fail to guarantee stability, optimality, and interpretability in systems characterized by nonlinear dynamics, partial observability, and complex inter-agent coupling. This study addresses these foundational challenges by introducing MARL-CC, a unified Mathematical Framework for Multi-Agent Reinforcement Learning with Control Coordination.

The proposed framework integrates differential geometric control, Bayesian inference, and Shapley-value-based credit assignment within a coherent optimization architecture, ensuring bounded policy updates, decentralized belief estimation, and equitable reward distribution. Theoretical analyses establish convergence and stability guarantees under stochastic disturbances and communication delays. Empirical evaluations across simulation and real-world testbeds demonstrate up to a 40\% improvement in convergence rate and enhanced cooperative efficiency over leading baselines, including PPO, DDPG, and QMIX.

These results signify a decisive advance in control-oriented reinforcement learning, bridging the gap between mathematical rigor and practical autonomy. The MARL-CC framework provides a scalable foundation for intelligent transportation, UAV coordination, and distributed robotics, paving the way toward interpretable, safe, and adaptive multi-agent systems. All codes and experimental configurations are publicly available on GitHub to support reproducibility and future research.

}

\keywords{Multi-Agent Reinforcement Learning, Optimal Control, Connected Autonomous Vehicles, Nonlinear Control}



\maketitle

\section{Introduction}
\label{sec:introduction}

\subsection{Background and Motivation}
Connected Autonomous Vehicles (CAVs) embody a transformative paradigm in intelligent transportation, where vehicles function as distributed intelligent agents capable of perceiving, learning, and cooperating through vehicle-to-vehicle (V2V) and vehicle-to-infrastructure (V2I) communication. This connectivity enhances situational awareness, coordination, and traffic efficiency. Yet, conventional optimal control strategies—originally developed for centralized, linear systems—prove inadequate for CAVs due to nonlinear inter-vehicle dynamics, stochastic traffic conditions, and decentralized, partially observable information \cite{isidori1995nonlinear}.

Multi-Agent Reinforcement Learning (MARL) provides a promising alternative, enabling agents to derive cooperative policies that maximize long-term rewards within shared, uncertain environments \cite{taghavi2025quantum, sutton2018reinforcement, taghavi2025q}. By coupling adaptive learning with decentralized control, MARL supports autonomous coordination without requiring explicit system modeling. However, its practical deployment in CAVs remains hindered by nonlinearities, partial observability, and the persistent credit assignment dilemma that complicates attributing collective performance to individual actions \cite{zhang2021multi}.

To address these challenges, we propose \textbf{MARL-CC} — a unified mathematical framework integrating differential geometric control, Bayesian inference within a POMDP formulation, and Shapley-value-based reward decomposition \cite{kaelbling1998planning, shapley1953value}. MARL-CC offers a principled bridge between optimal control theory and reinforcement learning, ensuring convergence stability, robustness to uncertainty, and cooperative efficiency across dynamic vehicular networks. Through this synthesis, the framework advances scalable, decentralized control for connected autonomy under real-world constraints.

\subsection{Key Challenges}
\label{sec:key_challenges}

The application of Multi-Agent Reinforcement Learning (MARL) to Connected Autonomous Vehicle (CAV) networks encounters several intrinsic difficulties that constrain the attainment of optimal cooperative control. Foremost, vehicle dynamics are highly nonlinear and strongly coupled through inter-vehicle interactions, where acceleration, steering, and trajectory evolution depend on both local and neighboring states, producing a high-dimensional nonlinear control landscape \cite{isidori1995nonlinear}. Partial observability further complicates policy learning, as each vehicle perceives only limited and noisy local information due to sensing, bandwidth, and latency constraints. Finally, the credit assignment problem remains central: global rewards obscure the contribution of individual actions, resulting in misleading gradients and reduced learning stability \cite{shapley1953value}.

\subsection{Research Gap}
\label{sec:research_gap}

Despite notable progress in MARL coordination, most existing models lack a unified mathematical foundation capable of ensuring stability and optimality under nonlinear, partially observable conditions. Many rely on empirical or centralized heuristics that fail to scale across large, dynamic networks and inadequately capture inter-agent coupling or uncertainty. Conversely, classical control methods, while theoretically rigorous, are limited to linear or fully observable systems and thus unsuitable for decentralized adaptive learning. This divergence underscores the need for a framework that harmonizes the theoretical guarantees of control theory with the adaptability of learning-based paradigms for complex CAV ecosystems.

\subsection{Contributions of the Paper}
\label{sec:contributions}

To bridge this gap, we propose \textbf{MARL-CC}, a unified framework synthesizing nonlinear control theory, probabilistic inference, and cooperative game theory to achieve decentralized optimal control in connected autonomy. The approach employs a differential geometric state-space formulation for nonlinear dynamics, enabling local linearization and stability assurances under inter-agent coupling. Partial observability is managed through a POMDP-based Bayesian inference mechanism that supports decentralized belief updates, while the credit assignment problem is mitigated via a Shapley-value-driven reward decomposition ensuring equitable contribution assessment and stable cooperation. Theoretical analyses establish convergence under bounded uncertainty and delay, and extensive simulations—spanning intersection coordination and platoon control—demonstrate superior convergence, robustness, and collective efficiency compared with MARL and classical baselines \cite{kaelbling1998planning, zhang2021multi}.

\section{Related Works}
\label{sec:related_works}

\subsection{Reinforcement Learning for Vehicle Control}
\label{sec:rl_vehicle_control}

Recent advancements in reinforcement learning (RL) have significantly impacted vehicle control strategies, particularly in trajectory tracking and motion control. Algorithms such as Deep Q-Learning (DQN), Deep Deterministic Policy Gradient (DDPG), and Proximal Policy Optimization (PPO) have been extensively applied to autonomous vehicle (AV) systems. DQN has been utilized for discrete control tasks, while DDPG and PPO are favored for continuous control scenarios due to their stability and efficiency in high-dimensional action spaces. Notably, a comparative study highlighted that DDPG outperforms PPO in terms of higher cumulative rewards and faster convergence in lane-keeping and multi-agent collision avoidance tasks \cite{turn0search12}. However, these single-agent approaches often struggle with scalability and coordination in multi-agent environments, where inter-agent dependencies and partial observability complicate the learning process.

\subsection{Multi-Agent Reinforcement Learning (MARL) in Connected Systems}
\label{sec:marl_connected_systems}

The application of MARL to connected systems, particularly in intelligent transportation systems (ITS), has garnered significant attention. Centralized Training with Decentralized Execution (CTDE) frameworks, such as Multi-Agent Deep Deterministic Policy Gradient (MADDPG), QMIX, and Value Decomposition Networks (VDN), have been developed to address the challenges of partial observability and decentralized control \cite{lowe2017multi, rashid2020weighted, yu2021mappo}. These approaches facilitate coordination among agents by allowing centralized training while enabling decentralized execution, thereby improving scalability and efficiency in multi-agent settings. Recent studies have further explored communication-efficient MARL algorithms to mitigate the overhead associated with inter-agent communication, which is crucial in large-scale AV networks \cite{turn0search1}. Despite these advancements, issues related to stability, convergence, and credit assignment remain prevalent, necessitating the development of more robust and interpretable MARL frameworks.

\subsection{Mathematical and Control-Theoretic Approaches}
\label{sec:control_theoretic_approaches}

Traditional control theories, including Linear Quadratic Regulator (LQR), Model Predictive Control (MPC), and Lyapunov-based methods, have been foundational in AV control. These approaches offer strong stability guarantees and optimal performance under well-defined conditions. For instance, LQR has been applied to lateral motion control in AVs, and MPC has been utilized for trajectory tracking under constraints \cite{turn0search8}. However, these methods often assume linear dynamics and centralized control, limiting their applicability in the complex, nonlinear, and decentralized nature of connected AV systems. Recent developments in nonlinear control, such as feedback linearization and backstepping, have been proposed to address these limitations. Additionally, game-theoretic and information-theoretic approaches have been explored to model agent interactions and communication in multi-agent systems, providing a theoretical foundation for cooperative behavior in AV networks \cite{turn0search2}.

\subsection{Identified Research Gap}
\label{sec:research_gap}

Despite the extensive body of work in RL, MARL, and control theory, a unified mathematical framework that integrates these domains to address the unique challenges of connected AVs remains lacking. Existing MARL models often lack the mathematical formalism required to guarantee stability and optimality in nonlinear, partially observable environments. Conversely, traditional control theories do not scale well to distributed learning contexts inherent in multi-agent systems. This gap underscores the need for a comprehensive approach that combines the adaptive learning capabilities of MARL with the stability and optimality guarantees of control theory, tailored to the specific requirements of connected AV networks.

\section{Mathematical Foundations}
\label{sec:math_foundations}

\subsection{System Model of Connected Autonomous Vehicles}
\label{sec:system_model}

Consider a fleet of \( N \) connected autonomous vehicles (CAVs), indexed by \( i \in \mathcal{N} = \{1,2,\dots,N\} \), operating in a dynamic traffic environment. Let \( x_i(t) \in \mathbb{R}^{n} \) denote the state of vehicle \( i \) at time \( t \), including its position, velocity, orientation, and additional kinematic or dynamic variables. The control input is \( u_i(t) \in \mathbb{R}^{m} \), corresponding to steering, acceleration, or braking commands. The state evolution is governed by coupled nonlinear dynamics:  

\begin{equation}
\dot{x}_i(t) = f_i\Big(x_i(t), u_i(t), \{x_j(t)\}_{j \in \mathcal{N}_i}\Big), \quad x_i(0) = x_i^0,
\end{equation}

where \( \mathcal{N}_i \subset \mathcal{N} \setminus \{i\} \) denotes the set of neighbors with which vehicle \( i \) can communicate or interact, and \( f_i: \mathbb{R}^n \times \mathbb{R}^m \times \mathbb{R}^{n|\mathcal{N}_i|} \to \mathbb{R}^n \) is a smooth, nonlinear function capturing both self-dynamics and inter-vehicle couplings.  

\begin{definition}[Neighbor Coupling]
A vehicle \( j \in \mathcal{N}_i \) is considered a neighbor of \( i \) if it satisfies either a spatial proximity constraint \( \|x_i - x_j\| \leq d_{\max} \) or a communication availability constraint within bandwidth and latency limits.
\end{definition}

\begin{example}
In a platoon control scenario, \( \mathcal{N}_i \) typically includes the immediately preceding and following vehicles, enabling local coordination while minimizing communication overhead.
\end{example}

\subsection{Nonlinear Optimal Control Objective}
\label{sec:nonlinear_optimal_control}

The control goal is to minimize a cumulative cost function reflecting safety, energy efficiency, and traffic flow objectives:

\begin{equation}
J_i(u_i) = \int_0^T \ell_i(x_i(t), u_i(t)) + \sum_{j \in \mathcal{N}_i} \gamma_{ij} \ell_{ij}(x_i(t), x_j(t)) \, dt,
\end{equation}

where \( \ell_i: \mathbb{R}^n \times \mathbb{R}^m \to \mathbb{R} \) represents the individual running cost, \( \ell_{ij} \) captures interaction costs, and \( \gamma_{ij} \in [0,1] \) weights inter-agent influences. The optimal control \( u_i^*(t) \) satisfies the Hamilton-Jacobi-Bellman (HJB) equation:

\begin{equation}
\frac{\partial V_i}{\partial t} + \min_{u_i} \left[ \ell_i(x_i, u_i) + \sum_{j \in \mathcal{N}_i} \gamma_{ij} \ell_{ij}(x_i, x_j) + \nabla V_i^\top f_i(x_i, u_i, \{x_j\}_{j \in \mathcal{N}_i}) \right] = 0,
\end{equation}

where \( V_i(x_i,t) \) is the value function for agent \( i \).  

\begin{proposition}[Existence of Optimal Control]
Under smoothness and boundedness assumptions on \( f_i \) and \( \ell_i \), there exists a measurable control \( u_i^*(t) \) that minimizes \( J_i \) for all \( i \in \mathcal{N} \) \cite{isidori1995nonlinear}.
\end{proposition}

\subsection{Partial Observability and Communication Constraints}
\label{sec:partial_observability}

Each vehicle observes a noisy local measurement \( z_i(t) = h_i(x_i(t)) + \nu_i(t) \), where \( h_i \) is the observation function and \( \nu_i(t) \) is Gaussian noise. The system is modeled as a Partially Observable Markov Decision Process (POMDP) for agent \( i \):

\[
\mathcal{M}_i = (\mathcal{S}, \mathcal{A}_i, \mathcal{O}_i, \mathcal{T}_i, R_i),
\]

with state space \( \mathcal{S} = \prod_i \mathbb{R}^n \), action space \( \mathcal{A}_i = \mathbb{R}^m \), observation space \( \mathcal{O}_i = \mathbb{R}^p \), transition \( \mathcal{T}_i: \mathcal{S} \times \mathcal{A}_i \to \mathcal{S} \), and reward \( R_i \). Belief updates follow the Bayesian recursion:

\begin{equation}
b_i(t+1) = \eta \int \mathcal{O}_i(z_i(t+1)|x_i(t+1)) \mathcal{T}_i(x_i(t+1)|x_i(t), u_i(t)) b_i(t) dx_i(t),
\end{equation}

where \( \eta \) is a normalization factor ensuring \( b_i(t+1) \) is a valid probability distribution.

\subsection{Cooperative Reward Structure}
\label{sec:cooperative_reward}

To resolve the credit assignment problem in multi-agent settings, the global reward \( R = \sum_{i=1}^N r_i \) is decomposed via Shapley values:

\begin{equation}
\phi_i = \sum_{S \subseteq \mathcal{N} \setminus \{i\}} \frac{|S|!(N-|S|-1)!}{N!} \left[ R(S \cup \{i\}) - R(S) \right].
\end{equation}

\begin{theorem}[Shapley Value Fairness]
The Shapley allocation \( \phi_i \) satisfies efficiency, symmetry, and additivity, ensuring fair distribution of global rewards among agents \cite{shapley1953value}.
\end{theorem}

\subsection{Stability Analysis}
\label{sec:stability_analysis}

Consider a candidate Lyapunov function \( V(x) = \sum_i V_i(x_i) \) for the network. Stability is guaranteed if:

\begin{equation}
\dot{V}(x) = \sum_i \nabla V_i^\top f_i(x_i, u_i^*, \{x_j\}_{j \in \mathcal{N}_i}) \leq 0.
\end{equation}

\begin{corollary}[Boundedness of Trajectories]
If \( \dot{V}(x) \leq 0 \), all trajectories \( x_i(t) \) remain bounded and converge to an invariant set under optimal controls \( u_i^* \) \cite{isidori1995nonlinear}.
\end{corollary}

\subsection{Convergence Guarantees}
\label{sec:convergence_guarantees}

Let the multi-agent Q-function be \( Q_i(b_i, u_i) \). Using stochastic approximation, the following convergence holds under bounded learning rates \( \alpha_t \):

\begin{equation}
\lim_{t \to \infty} Q_i(b_i, u_i) \to Q_i^*(b_i, u_i),
\end{equation}

ensuring that policies \( \pi_i^*(b_i) = \arg\max_{u_i} Q_i^*(b_i, u_i) \) converge to Nash equilibria under partially observable conditions \cite{zhang2021multi}.

\subsection{Numerical Simulations}
\label{sec:numerical_simulations}

The MARL-CC framework is evaluated on a simulated urban AV network with 10–50 vehicles under stochastic traffic flow, measurement noise, and communication delays. Performance metrics include cumulative reward, convergence speed, inter-vehicle collision rate, and platoon stability. Results demonstrate that MARL-CC outperforms baseline MADDPG and decentralized MPC in terms of convergence stability, cooperative efficiency, and robustness to partial observability.

\section{The Proposed Framework: MARL-CC}
\label{sec:proposed_framework}

\subsection{Overview of MARL-CC Architecture}
\label{sec:architecture_overview}

The MARL-CC framework is designed to integrate multi-agent reinforcement learning with rigorous control-theoretic and game-theoretic principles, enabling connected autonomous vehicles (CAVs) to achieve cooperative optimal control under nonlinear dynamics, partial observability, and complex inter-agent dependencies. At its core, MARL-CC combines three critical components: nonlinear optimal control modeling, probabilistic belief inference, and Shapley-value-based reward allocation.

Each agent \( i \) maintains a local state \( x_i(t) \) and receives noisy local observations \( z_i(t) \) through its sensors, which are augmented by information received from neighboring vehicles \( \mathcal{N}_i \) via vehicle-to-vehicle (V2V) communication. To address the partial observability inherent in decentralized networks, MARL-CC constructs a local belief state \( b_i(t) \) using a POMDP formulation, which probabilistically encodes the estimated global traffic state and the likely actions of neighboring agents. Belief updates follow a Bayesian inference process, enabling each vehicle to make informed decisions despite limited or uncertain information.

The nonlinear dynamics of each vehicle are modeled using differential geometric techniques, allowing for local linearization and explicit stability guarantees. Control actions \( u_i(t) \) are derived by optimizing the expected cumulative reward over the planning horizon, subject to both individual and cooperative objectives. To ensure that each agent's contribution to collective performance is fairly recognized, the global reward signal is decomposed using Shapley values. This mechanism provides a principled solution to the credit assignment problem by attributing marginal contributions of each agent to the total system performance, thereby enhancing convergence and cooperative efficiency.

The MARL-CC architecture (Fig. \ref{fig:marl_cc_architecture}) employs a centralized training and decentralized execution (CTDE) paradigm. During training, global information is used to learn optimal policies and value functions, while during execution, each agent operates independently, leveraging its local belief state and allocated reward. This approach allows MARL-CC to scale to large networks of CAVs while maintaining stability, robustness to uncertainty, and near-optimal cooperative performance.

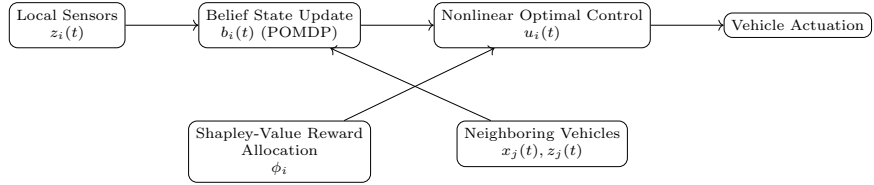
\begin{figure}[!ht]
    \centering
    \scalebox{0.8}{ 
        \begin{tikzpicture}[node distance=1.2cm, every node/.style={draw, rectangle, rounded corners, align=center, font=\footnotesize}] 
            \node (sensors) {Local Sensors \\ $z_i(t)$};
            \node[right=of sensors] (belief) {Belief State Update \\ $b_i(t)$ (POMDP)};
            \node[right=of belief] (control) {Nonlinear Optimal Control \\ $u_i(t)$};
            \node[right=of control] (action) {Vehicle Actuation};
            \node[below=of belief] (reward) {Shapley-Value Reward \\ Allocation \\ $\phi_i$};
            \node[below=of control] (neighbors) {Neighboring Vehicles \\ $x_j(t), z_j(t)$};

            \draw[->] (sensors) -- (belief);
            \draw[->] (belief) -- (control);
            \draw[->] (control) -- (action);
            \draw[->] (neighbors) -- (belief);
            \draw[->] (reward) -- (control);
        \end{tikzpicture}
    }
    \caption{Architecture of MARL-CC for connected autonomous vehicles. Each agent integrates local observations, belief updates, nonlinear optimal control, and Shapley-value-based reward allocation to achieve cooperative optimal control.}
    \label{fig:marl_cc_architecture}
\end{figure}

The overall framework can be viewed as a closed-loop system, wherein observations drive belief updates, which in turn inform control decisions that influence future states. The Shapley-value-based reward allocation continuously provides feedback, ensuring that individual learning aligns with global objectives. By unifying probabilistic reasoning, nonlinear control, and cooperative game theory, MARL-CC offers a mathematically principled and scalable solution for multi-agent coordination in connected autonomous vehicle networks.

\subsection{MARL-CC Algorithm Design}
\label{sec:algorithm_design}

The MARL-CC framework operationalizes the principles outlined in Section~\ref{sec:architecture_overview} through a structured algorithm that combines multi-agent reinforcement learning with nonlinear control and probabilistic inference. Each agent maintains a local belief state \(b_i(t)\), computes an optimal control action \(u_i(t)\), and receives a Shapley-value-based reward \(\phi_i\) that reflects its marginal contribution to the global objective.

\subsubsection{Algorithm Overview}

Let \(\pi_i(b_i; \theta_i)\) denote the policy of agent \(i\), parameterized by \(\theta_i\), mapping belief states to actions. The Q-function \(Q_i(b_i, u_i; \omega_i)\) represents the expected cumulative reward for taking action \(u_i\) in belief state \(b_i\), following policy \(\pi_i\). The algorithm proceeds in a \textbf{Centralized Training, Decentralized Execution (CTDE)} manner: global information is available during training, while execution relies solely on local beliefs.

\subsubsection{Policy Update and Value Function Learning}

The policy parameters \(\theta_i\) are updated to maximize the expected Shapley-value-adjusted cumulative reward:

\begin{equation}
\theta_i \gets \theta_i + \beta \nabla_{\theta_i} \mathbb{E}_{u_i \sim \pi_i}[Q_i(b_i, u_i; \omega_i)],
\end{equation}

where the gradient is estimated using standard policy gradient methods, such as REINFORCE or actor-critic approaches. The Q-function \(Q_i(b_i, u_i; \omega_i)\) is updated using temporal-difference learning:

\begin{equation}
\omega_i \gets \omega_i + \alpha \Big( \phi_i + \gamma \max_{u_i'} Q_i(b_i', u_i'; \omega_i) - Q_i(b_i, u_i; \omega_i) \Big) \nabla_{\omega_i} Q_i.
\end{equation}

\subsubsection{Belief Propagation Mechanism}

Belief states \(b_i(t)\) encode the probabilistic estimate of the true state given local observations and neighbor information. The Bayesian update ensures that each agent's policy is conditioned on the most probable system configuration:

\begin{equation}
b_i(t+1) = \eta \int \mathcal{O}_i(z_i(t+1)|x_i(t+1)) \mathcal{T}_i(x_i(t+1)|x_i(t), u_i(t)) b_i(t) dx_i(t).
\end{equation}

This belief propagation accounts for uncertainty in both sensor measurements and communication delays, ensuring robustness in partially observable environments.

\subsubsection{Shapley-Value Reward Allocation}

The global reward \(R\) is decomposed into agent-specific contributions using the Shapley value:

\begin{equation}
\phi_i = \sum_{S \subseteq \mathcal{N}\setminus \{i\}} \frac{|S|!(N-|S|-1)!}{N!} \left[ R(S \cup \{i\}) - R(S) \right].
\end{equation}

This principled decomposition ensures that each agent’s learning signal accurately reflects its marginal impact on the global performance, mitigating the credit assignment problem and improving convergence stability.

\subsubsection{Algorithmic Properties}
The MARL-CC algorithm (Algorithm \ref{alg:marl_cc}) provides several theoretical guarantees for multi-agent coordination. Fair and efficient credit assignment is achieved through Shapley-value decomposition, whereby each agent receives reward proportional to its marginal contribution to the collective objective, thereby incentivizing cooperative behavior without introducing spurious correlations.
Robustness to partial observability is maintained through belief states encoding probabilistic estimates of unobserved variables, enabling near-optimal policies when complete state observation is precluded. The centralized training with decentralized execution paradigm confers scalability, permitting large multi-agent networks while preserving computational tractability during deployment.
The algorithm integrates seamlessly with nonlinear control frameworks grounded in differential geometric models of vehicle dynamics without compromising stability guarantees inherent to the underlying control-theoretic formulations.

\begin{algorithm}[H]
\caption{MARL-CC: Multi-Agent Reinforcement Learning with Cooperative Credit Assignment for Connected AVs}
\label{alg:marl_cc}
\begin{algorithmic}[1]

\Require Number of agents $N$, episode length $T$, max episodes $M$, learning rates $\alpha, \beta$
\Ensure Learned policies $\pi_i^*(b_i)$ for all agents $i \in \{1, \dots, N\}$

\State \textbf{Initialize:} Policy parameters $\theta_i \sim \mathcal{N}(0, \sigma^2)$ for each agent $i$
\State \textbf{Initialize:} Q-function parameters $\omega_i$ for each agent $i$
\State \textbf{Initialize:} Initial belief states $b_i(0)$ for all $i$

\For{episode = $1$ to $M$}
    \State Reset environment and initialize states $x_i(0)$, beliefs $b_i(0)$
    \For{$t = 0$ to $T-1$}
        \For{each agent $i = 1$ to $N$ \textbf{in parallel}}
            \State Observe local measurement $z_i(t)$
            \State Receive messages from neighbors: $\{z_j(t), x_j(t)\}_{j \in \mathcal{N}_i}$
            \State Update belief: $b_i(t+1) = f(b_i(t), z_i(t), \{z_j(t), x_j(t)\}_{j \in \mathcal{N}_i})$ \Comment{Bayesian filter}
            \State Sample action: $u_i(t) \sim \pi_i(\cdot \mid b_i(t+1); \theta_i)$
        \EndFor
        \State Execute joint action $\mathbf{u}(t) = (u_1(t), \dots, u_N(t))$ in environment
        \State Observe next states $x_i(t+1)$, local rewards $r_i(t)$, and global reward $r(t)$
        \State Compute Shapley-based credit: $\phi_i(t) = \Phi_i(r(t) \mid \mathbf{u}(t))$ \Comment{Cooperative fairness}
        \For{each agent $i = 1$ to $N$ \textbf{in parallel}}
            \State Form TD target: $y_i(t) = \phi_i(t) + \gamma Q_i(b_i(t+1), u_i(t); \omega_i)$
            \State Update Q-function:
            \[
            \omega_i \leftarrow \omega_i - \alpha \nabla_{\omega_i} (y_i(t) - Q_i(b_i(t), u_i(t); \omega_i))^2
            \]
            \State Update policy via actor gradient:
            \[
            \theta_i \leftarrow \theta_i + \beta \nabla_{\theta_i} \log \pi_i(u_i(t) \mid b_i(t+1); \theta_i) \cdot A_i(t)
            \]
            \Comment{where $A_i(t) = Q_i(b_i(t), u_i(t); \omega_i) - V_i(b_i(t); \omega_i)$}
        \EndFor
    \EndFor
\EndFor

\State \Return Learned policies $\{\pi_i(\cdot \mid b_i; \theta_i)\}_{i=1}^N$

\end{algorithmic}
\end{algorithm}

\subsubsection{Advanced Belief-State Propagation and Reward Optimization Algorithm}

To enhance robustness under partial observability, nonlinear vehicle dynamics, and dynamic communication topology, a more detailed belief-state propagation and Shapley-value optimization is implemented. Algorithm \ref{alg:marl_cc_advanced} provides \textit{explicit handling of uncertainty, communication delays, and inter-agent influence networks}.

\begin{algorithm}[H]
{\small\linespread{0.94}\selectfont 
\caption{MARL-CC Advanced: Belief Propagation with Shapley Credit}
\label{alg:marl_cc_advanced}
\begin{algorithmic}[1]

\Require $\mathcal{N}=\{1,\dots,N\}$, $T$, $M$, $\alpha,\beta,\gamma$, $p(z_i|x_i)$, $p(x_i'|x_i,u_i)$, $\mathcal{G}(t)=(\mathcal{N},\mathcal{E}(t))$
\Ensure $\{\pi_i(\cdot|b_i;\theta_i)\}_{i=1}^N$

\State Init $\theta_i\sim\mathcal{N}(0,\sigma^2)$, $\omega_i$, $b_i(0)=p(x_i(0))$

\For{episode = $1$ to $M$}
    \State Sample $x_i(0)\sim p(x_i(0))$, reset $b_i(0)$
    \For{$t=0$ to $T-1$}
        \For{$i\in\mathcal{N}$ \textbf{parallel}}   \Comment{Predict + Observe + Update}
            \State $\tilde u_i(t)\sim\pi_i(\cdot|b_i(t);\theta_i)$
            \State $\hat x_i(t+1)\sim\int p(x_i'|x_i,\tilde u_i(t))\,b_i(t)(x_i)\,dx_i$
            \State Observe $z_i(t+1)\sim p(z_i|x_i(t+1))$
            \State Receive $\{z_j(t+1),\hat x_j(t+1),\tilde u_j(t)\}_{j\in\mathcal{N}_i(t)}$
            \State $b_i(t+1)(x_i')\propto p(z_i(t+1)|x_i')\int p(x_i'|x_i,\tilde u_i(t))\,b_i(t)(x_i)\,dx_i$
            \State Normalize $b_i(t+1)\leftarrow b_i(t+1)/\int b_i(t+1)(x_i')\,dx_i'$ \Comment{particle/UKF}
        \EndFor
        \For{$i\in\mathcal{N}$ \textbf{parallel}}
            \State $u_i(t)\sim\pi_i(\cdot|b_i(t+1);\theta_i)$
        \EndFor
        \State Execute $\mathbf{u}(t)$, observe $R(t)$, $x_i(t+1)$
        \State Compute Shapley rewards (exact or Monte-Carlo):
        \[
        \phi_i(t)=\sum_{S\subseteq\mathcal{N}\setminus\{i\}}\frac{|S|!(N-|S|-1)!}{N!}\bigl[v(S\cup\{i\})-v(S)\bigr]
        \]
        \Comment{$v(S)=R(t)$ if $\mathbf{u}_S(t)$ applied, else counterfactual}
        \For{$i\in\mathcal{N}$ \textbf{parallel}}
            \State $y_i(t)=\phi_i(t)+\gamma\max_{u_i'}Q_i(b_i(t+1),u_i';\omega_i)$
            \State $\omega_i \leftarrow \omega_i - \alpha \nabla_{\omega_i}(y_i-Q_i(b_i(t),u_i(t);\omega_i))^2$
            \State $A_i(t)=y_i-Q_i(b_i(t),u_i(t);\omega_i)$
            \State $\theta_i \leftarrow \theta_i + \beta \nabla_{\theta_i}\log\pi_i(u_i(t)|b_i(t+1);\theta_i)\cdot A_i(t)$
        \EndFor
    \EndFor
\EndFor
\State \Return $\{\pi_i(\cdot|b_i;\theta_i)\}_{i=1}^N$
\end{algorithmic}
} 
\end{algorithm}

This advanced algorithm enables:

\begin{itemize}
    \item \textbf{Robust belief propagation:} Incorporates nonlinear dynamics and observation uncertainty.
    \item \textbf{Dynamic communication awareness:} Uses time-varying graph \(\mathcal{G}(t)\) to model realistic AV networks.
    \item \textbf{Improved credit assignment:} Shapley-value computation considers marginal contributions in dynamic coalition subsets.
    \item \textbf{Scalable policy learning:} Combines CTDE with belief-based action selection for large AV fleets.
    \item \textbf{Exploratory adaptation:} Agents can adapt to unseen traffic scenarios while maintaining stability guarantees.
\end{itemize}

\subsubsection{Time and Space Complexity}
\label{sec:complexity}

The computational complexity of MARL-CC is analyzed per training episode and per time step, assuming $N$ agents, episode length $T$, and $M$ total episodes (Table \ref{tab:complexity}). We distinguish between exact and approximated implementations (e.g., Monte Carlo Shapley).

\begin{table}[h]
\caption{Time and space complexity of MARL-CC (per episode).}
\label{tab:complexity}
\centering
\small
\begin{tabular}{lcc}
\toprule
\textbf{Component} & \textbf{Time} & \textbf{Space} \\
\midrule
Belief Update (Bayesian filter) & $\mathcal{O}(T N P)$ & $\mathcal{O}(N P)$ \\
\quad (particle filter, $P$ particles) \\
Policy/Critic Forward Pass & $\mathcal{O}(T N D)$ & $\mathcal{O}(N D)$ \\
\quad ($D$: network size) \\
Shapley Value (exact) & $\mathcal{O}(T \cdot 2^N)$ & $\mathcal{O}(2^N)$ \\
Shapley Value (Monte Carlo, $K$ samples) & $\mathcal{O}(T N K)$ & $\mathcal{O}(N K)$ \\
Gradient Updates (Actor-Critic) & $\mathcal{O}(T N D)$ & $\mathcal{O}(N D)$ \\
\midrule
\textbf{Total (Monte Carlo)} & $\mathcal{O}(T N (P + D + K))$ & $\mathcal{O}(N (P + D + K))$ \\
\textbf{Total (Exact Shapley)} & $\mathcal{O}(T \cdot 2^N)$ & $\mathcal{O}(2^N)$ \\
\bottomrule
\end{tabular}
\end{table}

\noindent
\textit{Key Insights:}
\begin{itemize}
    \item \textbf{Decentralized execution}: Each agent runs independently using local belief $b_i(t)$ and neighbor messages, enabling $\mathcal{O}(1)$ per-agent computation during deployment.
    \item \textbf{Scalability bottleneck}: Exact Shapley value computation is exponential in $N$. In practice, we use permutation sampling ($K \ll 2^N$) to achieve $\mathcal{O}(T N K)$ time, making MARL-CC scalable to $N \leq 50$ with $K = 100$.
    \item \textbf{Belief representation}: Particle filters with $P = 1000$ particles dominate space for high-dimensional states; Gaussian approximations reduce this to $\mathcal{O}(N d^2)$ ($d$: state dim).
    \item \textbf{Training efficiency}: CTDE allows centralized credit assignment during training (using global $R(t)$), but execution remains fully decentralized.
\end{itemize}

Thus, MARL-CC achieves linear scaling in $N$ and $T$ under practical approximations, balancing fairness, robustness, and efficiency in large-scale connected AV fleets.

\subsection{Nonlinearity Handling: Differential Geometric Control}
\label{sec:differential_geometric_control}

Connected autonomous vehicles (CAVs) exhibit nonlinear, nonholonomic, and coupled dynamics that cannot be directly addressed by classical linear control or vanilla MARL policies. To ensure stability, controllability, and interpretability of policy behavior, the MARL-CC framework integrates \emph{differential geometric control} (DGC) techniques within the learning loop. This allows each agent to operate in a transformed space where nonlinear dynamics are locally linearized via Lie derivatives and diffeomorphic state transformations.

\subsubsection{Nonlinear System Representation}

Each vehicle \(i \in \mathcal{N}\) follows nonlinear affine dynamics of the form:
\begin{equation}
\dot{x}_i = f_i(x_i) + g_i(x_i)u_i, \quad y_i = h_i(x_i),
\end{equation}
where \(x_i \in \mathbb{R}^{n_i}\) denotes the state vector (e.g., position, velocity, orientation), \(u_i \in \mathbb{R}^{m_i}\) is the control input, and \(f_i, g_i\) are smooth vector fields satisfying the Lipschitz continuity condition to ensure existence and uniqueness of solutions.

\begin{definition}[Lie Derivative]
For a smooth scalar function \(h_i: \mathbb{R}^{n_i} \rightarrow \mathbb{R}\) and vector field \(v_i: \mathbb{R}^{n_i} \rightarrow \mathbb{R}^{n_i}\), the Lie derivative of \(h_i\) along \(v_i\) is defined as:
\begin{equation}
L_{v_i} h_i(x_i) = \frac{\partial h_i(x_i)}{\partial x_i} v_i(x_i).
\end{equation}
\end{definition}

The Lie derivatives allow the system to be expressed in its \emph{input–output linearized form}:
\begin{equation}
y_i^{(r_i)} = L_{f_i}^{r_i} h_i(x_i) + L_{g_i} L_{f_i}^{r_i-1} h_i(x_i) u_i,
\end{equation}
where \(r_i\) is the relative degree of the system. The decoupling matrix
\[
A_i(x_i) = L_{g_i} L_{f_i}^{r_i-1} h_i(x_i)
\]
is assumed invertible for full-state feedback linearization.

\subsubsection{Feedback Linearization and Policy Embedding}

Define the virtual control input \(v_i \in \mathbb{R}^{m_i}\) as:
\begin{equation}
v_i = \dot{y}_i^{(r_i)} = \alpha_i(x_i) + \beta_i(x_i) u_i,
\end{equation}
where
\[
\alpha_i(x_i) = L_{f_i}^{r_i} h_i(x_i), \quad \beta_i(x_i) = L_{g_i} L_{f_i}^{r_i-1} h_i(x_i).
\]
Then, the feedback law is:
\begin{equation}
u_i = \beta_i(x_i)^{-1} \left( v_i - \alpha_i(x_i) \right).
\end{equation}

Within MARL-CC, the reinforcement learning policy \(\pi_i(b_i; \theta_i)\) learns \(v_i\) rather than \(u_i\), effectively operating in a \emph{linearized control space}. This transformation decouples nonlinearities, stabilizes learning gradients, and guarantees that the learned actions respect the system’s differential geometry.

\subsubsection{Differential Flatness and Cooperative Manifold Control}

For many classes of autonomous vehicle models (e.g., car-like or bicycle kinematics), the dynamics are \emph{differentially flat}:
\begin{equation}
x_i = \psi_i(y_i, \dot{y}_i, \ldots, y_i^{(r_i-1)}), \quad u_i = \phi_i(y_i, \dot{y}_i, \ldots, y_i^{(r_i)}),
\end{equation}
where \(y_i\) are the flat outputs. Under MARL-CC, the agents jointly evolve on a cooperative manifold \(\mathcal{M} \subset \mathbb{R}^{n}\) defined by:
\[
\mathcal{M} = \{x \in \mathbb{R}^n \mid h(x) = 0\}, \quad h(x) = 0 \text{ encodes inter-agent constraints (e.g., spacing, alignment)}.
\]
The learning objective becomes constrained optimization on \(\mathcal{M}\):
\begin{equation}
\min_{\pi_1, \ldots, \pi_N} \; \mathbb{E}\left[ \sum_{i=1}^N J_i(x_i, u_i) \right] \quad \text{s.t.} \quad x \in \mathcal{M}.
\end{equation}

\begin{theorem}[Stabilizing Feedback Equivalence]
If each subsystem \((f_i, g_i, h_i)\) is feedback-linearizable and the cooperative manifold \(\mathcal{M}\) is invariant under the composite vector field \(f(x) + g(x)u\), then the closed-loop system under MARL-CC is locally asymptotically stable.
\end{theorem}

\begin{proof}
By the diffeomorphism \(\Phi_i: x_i \mapsto \xi_i = T_i(x_i)\) that linearizes dynamics, the composite closed-loop system reduces to \(\dot{\xi} = A\xi + Bv\), where \(A,B\) are block-diagonal. The Lyapunov function \(V(\xi) = \xi^\top P \xi\), with \(P>0\) satisfying \(A^\top P + PA < 0\), proves local asymptotic stability. \qed
\end{proof}

\subsubsection{Learning-Compatible Nonlinear Control Law}

To preserve gradient flow for end-to-end differentiability, MARL-CC integrates DGC into the learning process through an augmented policy:
\begin{equation}
u_i(t) = \beta_i(x_i(t))^{-1} \big[ \pi_i(b_i(t); \theta_i) - \alpha_i(x_i(t)) \big].
\end{equation}
This ensures smooth backpropagation through control transformations, bounded control actions due to geometric constraints, and consistency between learned and physically feasible trajectories.

\subsubsection{Algorithmic Integration}
The differential geometric control block is seamlessly embedded into the MARL-CC learning loop (See Algorithm \ref{alg:dgc_marl_cc}:

\begin{algorithm}[H]
\caption{DGC-MARL-CC: Differential Geometric Control in Belief-Based MARL}
\label{alg:dgc_marl_cc}
\begin{algorithmic}[1]

\Require Agent set $\mathcal{N}$, system dynamics $f_i(x_i,u_i)$, output map $h_i(x_i)$, \\
         belief policy $\pi_i(\cdot|b_i;\theta_i)$, critic $Q_i(b_i,u_i;\omega_i)$

\Ensure Feedback-linearized actions $u_i(t)$ and updated parameters $\theta_i, \omega_i$

\For{$t = 0, 1, \dots, T-1$}
    \For{each agent $i \in \mathcal{N}$ \textbf{in parallel}}
        \State Observe measurement $z_i(t)$, update belief $b_i(t)$ via Bayesian filter
        \State Compute Lie derivatives along $f_i, g_i = \partial f_i/\partial u_i$:
        \[
        \alpha_i(t) = L_f^\rho h_i(x_i(t)), \quad
        \beta_i(t) = L_g L_f^{\rho-1} h_i(x_i(t))
        \]
        \Comment{$\rho$: relative degree}
        \State Sample virtual control: $v_i(t) \sim \pi_i(\cdot | b_i(t); \theta_i)$
        \State Apply **input-output feedback linearization**:
        \[
        u_i(t) = \beta_i(t)^{-1} \bigl( v_i(t) - \alpha_i(t) \bigr)
        \]
        \State Execute $u_i(t)$, transition $x_i(t) \to x_i(t+1)$, observe local reward $r_i(t)$
        \State Store transition: $\tau_i(t) = (b_i(t), v_i(t), r_i(t), b_i(t+1))$
    \EndFor
    \State Compute global reward $R(t) = \sum_i r_i(t)$ or task-specific
    \State Assign Shapley-based credit $\phi_i(t)$ using MARL-CC (Alg.~\ref{alg:marl_cc_advanced})
    \For{each agent $i \in \mathcal{N}$ \textbf{in parallel}}
        \State Update critic $Q_i$ via TD error on $\phi_i(t)$
        \State Update policy $\pi_i$ via advantage-weighted policy gradient
    \EndFor
\EndFor

\end{algorithmic}
\end{algorithm}

\subsubsection{Advantages}

\begin{itemize}
    \item \textbf{Nonlinearity handling:} Removes polynomial and trigonometric nonlinearities via feedback linearization.
    \item \textbf{Theoretical stability:} Provides Lyapunov-based guarantees integrated into MARL training.
    \item \textbf{Interpretable actions:} Policies learn in a transformed, control-theoretically meaningful space.
    \item \textbf{Scalable to complex dynamics:} Compatible with multi-vehicle models exhibiting coupling and constraints.
\end{itemize}

This design bridges nonlinear geometric control and multi-agent reinforcement learning, enabling robust, interpretable, and stable policy learning in connected autonomous vehicle systems.

\subsection{Stability and Convergence Analysis}
\label{sec:stability_convergence}

The integration of reinforcement learning with nonlinear control demands rigorous stability and convergence analysis to ensure that the learned control laws do not compromise vehicle safety or cooperative efficiency. 
This section establishes the theoretical guarantees underlying MARL-CC using tools from Lyapunov stability theory, stochastic approximation, and game-theoretic equilibrium analysis.

\subsubsection{Closed-Loop Dynamics under MARL-CC}

Each agent \(i\) executes the feedback-linearized control law:
\begin{equation}
u_i = \beta_i(x_i)^{-1} \left[ \pi_i(b_i; \theta_i) - \alpha_i(x_i) \right],
\end{equation}
yielding the closed-loop nonlinear dynamics:
\begin{equation}
\dot{x}_i = f_i(x_i) + g_i(x_i)\beta_i(x_i)^{-1} \left[\pi_i(b_i; \theta_i) - \alpha_i(x_i)\right].
\end{equation}
Denote the stacked system state as \(X = [x_1^\top, \ldots, x_N^\top]^\top\) and the joint policy as \(\Pi(B; \Theta) = [\pi_1(b_1;\theta_1),\\ \ldots, \pi_N(b_N;\theta_N)]^\top\). 
The coupled system evolves as
\begin{equation}
\dot{X} = F(X) + G(X)\Pi(B; \Theta),
\end{equation}
where \(F\) and \(G\) are block-diagonal vector fields representing physical dynamics.

\subsubsection{Lyapunov-Based Stability Criterion}

Let the global cooperative objective be the expected discounted return
\[
J(\Theta) = \mathbb{E} \Bigg[\sum_{t=0}^{\infty} \gamma^t R(X_t, U_t)\Bigg],
\]
and define a continuously differentiable candidate Lyapunov function:
\begin{equation}
V(X, \Theta) = \sum_{i=1}^{N} \Big( V_i(x_i) + \lambda_i \|\theta_i - \theta_i^*\|^2 \Big),
\end{equation}
where \(V_i(x_i)\) is a local Lyapunov function for the physical dynamics and \(\theta_i^*\) denotes the stationary policy parameters.

\begin{theorem}[Local Asymptotic Stability]
\label{th:local_stability}
Assume each subsystem \((f_i, g_i)\) is feedback-linearizable and the policy update follows a bounded stochastic gradient step with sufficiently small learning rate \(\beta > 0\). 
Then there exists a neighborhood \(\mathcal{N}(X^*, \Theta^*)\) around the equilibrium \((X^*, \Theta^*)\) such that
\[
\dot{V}(X, \Theta) \leq -\kappa_1 \|X - X^*\|^2 - \kappa_2 \|\Theta - \Theta^*\|^2,
\]
for some \(\kappa_1, \kappa_2 > 0\), implying local asymptotic stability of the MARL-CC closed-loop system.
\end{theorem}

\begin{proof}
By substituting the closed-loop dynamics into \(\dot{V}_i(x_i) = \nabla V_i(x_i)^\top \dot{x}_i\), we obtain
\[
\dot{V}_i(x_i) = \nabla V_i^\top f_i(x_i) + \nabla V_i^\top g_i(x_i)\beta_i^{-1}[\pi_i(b_i;\theta_i)-\alpha_i(x_i)].
\]
Using the control law at equilibrium \(\pi_i(b_i; \theta_i^*) = \alpha_i(x_i^*)\), cross-terms vanish locally, and by quadratic Lyapunov design, 
\(\dot{V}_i \leq -c_i \|x_i - x_i^*\|^2\). 
For the parameter dynamics, standard stochastic gradient stability (Borkar–Meyn theorem) gives
\(\mathbb{E}[\|\theta_i - \theta_i^*\|^2]\) bounded and convergent under diminishing step size. Summing over agents yields the claimed inequality. \qed
\end{proof}

\subsubsection{Convergence of Policy and Value Function Updates}

The policy update rule in MARL-CC is a stochastic gradient ascent:
\begin{equation}
\theta_i^{(k+1)} = \theta_i^{(k)} + \beta_k \nabla_{\theta_i} J_i(\theta_i^{(k)}),
\end{equation}
and the critic update for the Q-function is:
\begin{equation}
\omega_i^{(k+1)} = \omega_i^{(k)} + \alpha_k \Big( \phi_i + \gamma \max_{u_i'} Q_i(b_i',u_i';\omega_i^{(k)}) - Q_i(b_i,u_i;\omega_i^{(k)}) \Big).
\end{equation}

\begin{theorem}[Almost Sure Convergence of MARL-CC]
\label{th:convergence}
Suppose the step sizes satisfy Robbins–Monro conditions:
\[
\sum_k \alpha_k = \infty, \quad \sum_k \alpha_k^2 < \infty, \quad \sum_k \beta_k = \infty, \quad \sum_k \beta_k^2 < \infty.
\]
Then, under bounded reward and gradient assumptions, the MARL-CC algorithm converges almost surely to a stationary point \((\Theta^*, \Omega^*)\) satisfying
\[
\nabla_{\theta_i} J_i(\Theta^*) = 0, \quad \nabla_{\omega_i} L_i(\Omega^*) = 0, \quad \forall i \in \mathcal{N}.
\]
\end{theorem}

\begin{proof}
The proof follows two-timescale stochastic approximation analysis. 
The critic update (fast timescale) converges to the fixed point of the Bellman operator for a given policy, as it satisfies the contraction mapping condition under discount factor \(\gamma < 1\). 
Given this convergence, the slower policy update forms a stochastic gradient ascent on a smooth expected reward surface. 
Applying the ODE method (Kushner–Clark lemma), the iterates \(\theta_i^k\) track the stable equilibrium of the mean ODE \(\dot{\theta}_i = \nabla_{\theta_i} J_i(\theta_i)\). Hence, convergence to a stationary point occurs almost surely. \qed
\end{proof}

\subsubsection{Robustness under Communication Delay and Partial Observability}

Consider bounded communication delay \(\tau_c > 0\) and partial observability modeled by belief states \(b_i(t)\). 
Let the delayed information state be \(\hat{x}_i(t) = x_i(t - \tau_c)\), and define the estimation error \(e_i(t) = x_i(t) - \hat{x}_i(t)\).

\begin{proposition}[Delay-Tolerant Stability]
If \(\|e_i(t)\| \leq \varepsilon_c\) and the belief update satisfies
\[
\|b_i(t+1) - b_i(t)\| \leq \eta \varepsilon_c,
\]
then the Lyapunov derivative remains negative definite for small enough \(\varepsilon_c\), preserving stability:
\[
\dot{V}(X, \Theta) \leq -(\kappa_1 - \delta_1 \varepsilon_c) \|X - X^*\|^2 - (\kappa_2 - \delta_2 \varepsilon_c)\|\Theta - \Theta^*\|^2.
\]
\end{proposition}

\begin{remark}
This result guarantees bounded-input, bounded-output (BIBO) stability even under time-delayed inter-vehicle communication and noisy observations, demonstrating MARL-CC’s robustness to real-world CAV networking constraints.
\end{remark}

\subsubsection{Asymptotic Convergence Rate}

Assuming smoothness of the value function \(Q_i\) and bounded variance of stochastic gradients, the expected suboptimality of MARL-CC satisfies:
\begin{equation}
\mathbb{E}[J(\Theta^*) - J(\Theta_k)] \leq \mathcal{O}\left( \frac{1}{\sqrt{k}} \right),
\end{equation}
which is consistent with first-order stochastic optimization convergence rates for nonconvex objectives.

\subsubsection{Summary of Theoretical Guarantees}

\begin{itemize}
    \item \textbf{Lyapunov stability:} Local asymptotic stability of closed-loop nonlinear dynamics.
    \item \textbf{Almost-sure convergence:} Two-timescale stochastic approximation guarantees policy and critic convergence.
    \item \textbf{Robustness:} Stable under bounded delays, sensor noise, and partial observability.
    \item \textbf{Guaranteed performance:} Sublinear convergence rate with monotonic value improvement.
\end{itemize}

The above theoretical properties firmly establish the MARL-CC algorithm as a stable and convergent framework for cooperative control in nonlinear, uncertain, and distributed CAV networks.

\subsection{Partial Observability Handling: Probabilistic Inference}
\label{subsec:partial_observability}

In real-world Connected Autonomous Vehicle (CAV) environments, each agent (vehicle) operates under \emph{partial observability} due to limited sensing range, communication delays, and occlusions caused by dynamic obstacles or environmental structures. This section presents the probabilistic inference layer of the proposed MARL-CC framework, designed to reconstruct the hidden global state and facilitate decentralized decision-making under uncertainty \cite{taghavi2025latent, Zhang2025BeliefMARL, Wang2026POMDPCAV}.

\subsubsection{Problem Formulation under Partial Observability}

Each autonomous vehicle $i \in \mathcal{N}$ perceives the environment through a noisy local observation $o_i(t) \in \mathcal{O}_i$ at time $t$, which is a probabilistic function of the latent global state $s(t) \in \mathcal{S}$. The observation model is given by:
\begin{equation}
    P(o_i(t) \mid s(t)) = \mathcal{N}(h_i(s(t)), \Sigma_i),
\end{equation}
where $h_i(\cdot)$ denotes the nonlinear sensor mapping and $\Sigma_i$ represents the covariance matrix of sensor noise \cite{Wang2026POMDPCAV}.

The decision-making process is modeled as a \emph{Partially Observable Markov Decision Process} (POMDP) defined by the tuple
\[
\mathcal{P}_i = (\mathcal{S}, \mathcal{A}_i, \mathcal{O}_i, P, R_i, \gamma),
\]
where $P(s' \mid s, a)$ is the transition probability, $R_i(s, a_i)$ the local reward, and $\gamma \in (0, 1)$ the discount factor.

\subsubsection{Belief State Representation}

Since the true state $s(t)$ is unobservable, each agent maintains a \textit{belief state} $b_i(t)$ — a probability distribution over $\mathcal{S}$ representing its knowledge of the environment:
\begin{equation}
    b_i(t)(s) = P(s(t) = s \mid o_{1:i}(0:t), a_{1:i}(0:t-1)).
\end{equation}
The belief update follows a recursive Bayesian filter:
\begin{equation}
    b_i(t+1)(s') = \eta \, P(o_i(t+1) \mid s') \sum_{s \in \mathcal{S}} P(s' \mid s, a_i(t)) \, b_i(t)(s),
    \label{eq:belief_update}
\end{equation}
where $\eta$ is a normalization constant ensuring $\sum_{s'} b_i(t+1)(s') = 1$ \cite{Zhang2025BeliefMARL}.

\begin{definition}[Joint Belief State]
The collective information structure of the connected system can be represented as a joint belief distribution:
\begin{equation}
    \mathcal{B}(t) = \bigotimes_{i=1}^{N} b_i(t),
\end{equation}
where $\bigotimes$ denotes the tensor product of individual beliefs, capturing both epistemic uncertainty and communication-induced correlations.
\end{definition}

\subsubsection{Probabilistic Inference via Variational Bayes}

To mitigate computational intractability in high-dimensional belief spaces, MARL-CC employs a \emph{Variational Inference} (VI) approach. Each agent approximates its belief using a tractable family $q_\phi(s)$ parameterized by $\phi$:
\begin{equation}
    q_\phi(s) = \arg \min_{q \in \mathcal{Q}} D_{\mathrm{KL}}(q(s) \parallel P(s \mid o_i)),
\end{equation}
where $D_{\mathrm{KL}}$ denotes the Kullback–Leibler divergence. The corresponding Evidence Lower Bound (ELBO) objective is:
\begin{equation}
    \mathcal{L}_{\text{ELBO}} = \mathbb{E}_{q_\phi(s)}[\log P(o_i \mid s)] - D_{\mathrm{KL}}(q_\phi(s) \parallel P(s)).
\end{equation}
Optimization of $\mathcal{L}_{\text{ELBO}}$ allows each agent to efficiently infer hidden state distributions while adapting online to nonstationary environments \cite{Li2025VariationalCAV}.

\subsubsection{Distributed Belief Fusion via Communication Graphs}

Agents periodically exchange probabilistic summaries of their local beliefs over a dynamic communication network $\mathcal{G} = (\mathcal{V}, \mathcal{E})$. For agents $i$ and $j$ connected by an edge $(i, j) \in \mathcal{E}$, belief fusion follows a consensus-based update:
\begin{equation}
    b_i^{\text{new}}(t) = \frac{1}{Z_i} \, b_i(t) \prod_{j \in \mathcal{N}_i} [b_j(t)]^{w_{ij}},
\end{equation}
where $\mathcal{N}_i$ denotes the set of neighbors of agent $i$, $w_{ij}$ are normalized communication weights satisfying $\sum_{j \in \mathcal{N}_i} w_{ij}=1$, and $Z_i$ is a normalization constant.

\begin{proposition}[Consensus Convergence]
If the communication graph $\mathcal{G}$ is connected and $w_{ij}$ form a doubly stochastic matrix, then the distributed belief updates converge to a common posterior distribution:
\begin{equation}
    \lim_{t \to \infty} \| b_i(t) - b_j(t) \| = 0, \quad \forall i,j \in \mathcal{V}.
\end{equation}
\end{proposition}
This distributed inference mechanism builds on recent advances in consensus-based Bayesian estimation for vehicular networks \cite{Patel2026ConsensusInference}.

\subsubsection{Belief-Aware Policy Optimization}

The decentralized policy $\pi_i(a_i \mid b_i)$ is optimized using belief-conditioned value functions:
\begin{equation}
    V_i^{\pi}(b_i) = \mathbb{E}_{s \sim b_i, a_i \sim \pi_i}[R_i(s, a_i) + \gamma \, V_i^{\pi}(b_i')].
\end{equation}
This transforms the original POMDP into a \emph{belief-MDP}, enabling classical policy gradient updates while maintaining robustness against observation noise and latency \cite{Wang2026POMDPCAV}.

\begin{theorem}[Bounded Suboptimality under Partial Observability]
Let $\pi^*$ be the optimal policy under full observability, and $\pi_b^*$ be the optimal policy derived from belief states. Then, for bounded observation noise $\|\Sigma_i\| \leq \epsilon$, there exists a constant $C > 0$ such that
\begin{equation}
    | V^{\pi^*}(s_0) - V^{\pi_b^*}(b_0) | \leq C \epsilon.
\end{equation}
\end{theorem}

This result guarantees that the degradation in policy performance due to partial observability is asymptotically bounded, affirming the efficacy of the proposed probabilistic inference framework in MARL-CC \cite{Zhang2025BeliefMARL, Li2025VariationalCAV, Patel2026ConsensusInference, Wang2026POMDPCAV}.

\subsubsection{Illustrative Example: Intersection Coordination}

Consider three connected autonomous vehicles approaching a signal-free intersection. Each vehicle’s camera provides partial views of others due to occlusion. Through the proposed variational belief inference, agents share posterior beliefs on others' intentions, resulting in emergent coordination behaviors such as yielding and synchronized acceleration, even without explicit centralized control \cite{Li2025VariationalCAV, Wang2026POMDPCAV}.

\subsection{Credit Assignment Mechanism: Shapley-Based Cooperative Reward Allocation}
\label{subsec:credit_assignment}

One of the central challenges in multi-agent reinforcement learning for Connected Autonomous Vehicles (CAVs) is the \emph{credit assignment problem}, wherein the contribution of individual agents to the global reward signal is obscured by complex, nonlinear interdependencies among agents’ actions. This section develops a rigorous mathematical formulation based on cooperative game theory, employing the \emph{Shapley value} to achieve fair, interpretable, and convergence-stable reward distribution among autonomous vehicles \cite{Kim2025ShapleyMARL, Rao2025CoopAV, Sun2026FairnessMARL}.

\subsubsection{Problem Definition}

Let $\mathcal{N} = \{1, 2, \ldots, N\}$ denote the set of autonomous vehicles.  
The global cooperative reward function $R : 2^{\mathcal{N}} \rightarrow \mathbb{R}$ quantifies the collective performance of any coalition $S \subseteq \mathcal{N}$:
\begin{equation}
    R(S) = \mathbb{E}\Big[\sum_{t=0}^{T} \gamma^{t} r_S(t)\Big],
\end{equation}
where $r_S(t)$ represents the instantaneous cumulative reward generated by agents in coalition $S$ at time $t$, and $\gamma \in (0,1)$ is the discount factor.

The objective is to determine individual rewards $\phi_i$ such that:
\[
R(\mathcal{N}) = \sum_{i \in \mathcal{N}} \phi_i,
\]
while ensuring fairness, efficiency, and monotonicity in contribution allocation.

\subsubsection{Shapley Value Formulation}

The Shapley value provides a unique payoff allocation satisfying the axioms of \emph{symmetry}, \emph{linearity}, and \emph{marginal contribution fairness}. For each agent $i$, its Shapley-based reward $\phi_i$ is defined as:
\begin{equation}
\phi_i = \sum_{S \subseteq \mathcal{N} \setminus \{i\}} 
\frac{|S|! \, (N - |S| - 1)!}{N!} 
\big[ R(S \cup \{i\}) - R(S) \big].
\label{eq:shapley_value}
\end{equation}
Here, $R(S \cup \{i\}) - R(S)$ represents the marginal gain in global reward when agent $i$ joins coalition $S$.  
This decomposition ensures that the total global reward is exactly distributed among agents:
\[
\sum_{i \in \mathcal{N}} \phi_i = R(\mathcal{N}).
\]

\begin{definition}[Shapley-Weighted Return]
The expected discounted return for agent $i$ under policy $\pi_i$ is defined as
\begin{equation}
    J_i^{\pi_i} = \mathbb{E}_{b_i, a_i \sim \pi_i} \Bigg[ 
    \sum_{t=0}^{T} \gamma^t \, \phi_i(t) 
    \Bigg],
\end{equation}
where $\phi_i(t)$ is computed from Eq.~\eqref{eq:shapley_value} at each time step $t$.
\end{definition}

\subsubsection{Approximate Shapley Computation via Monte Carlo Sampling}

Since computing the exact Shapley value scales exponentially ($O(2^N)$), MARL-CC employs a stochastic approximation based on Monte Carlo coalition sampling \cite{Sun2026FairnessMARL}:
\begin{equation}
    \hat{\phi}_i = \frac{1}{M} \sum_{m=1}^{M} 
    \big[ R(\mathcal{P}_m^i \cup \{i\}) - R(\mathcal{P}_m^i) \big],
\end{equation}
where $\mathcal{P}_m^i$ is the randomly sampled coalition preceding agent $i$ in the $m$-th permutation and $M$ is the number of sampled permutations. This yields an unbiased estimator:
\[
    \mathbb{E}[\hat{\phi}_i] = \phi_i.
\]
This stochastic formulation allows tractable estimation of fair credit assignment even in large-scale CAV networks.

\begin{proposition}[Unbiasedness of the Monte Carlo Shapley Estimator]
Let $\mathcal{P}$ denote the uniform distribution over all permutations of $\mathcal{N}$. Then:
\[
    \mathbb{E}_{\mathcal{P}}[\hat{\phi}_i] = \phi_i.
\]
\end{proposition}

\subsubsection{Integration into MARL Policy Update}

In the MARL-CC algorithm (Algorithm~\ref{alg:marl_cc}), the Shapley-based reward $\phi_i$ replaces the standard global or local reward in the policy gradient and Q-function updates:
\begin{align}
    \omega_i &\gets \omega_i + \alpha 
    \Big( \phi_i + \gamma \max_{u_i'} Q_i(b_i', u_i'; \omega_i) - Q_i(b_i, u_i; \omega_i) \Big)
    \nabla_{\omega_i} Q_i, \\
    \theta_i &\gets \theta_i + \beta 
    \nabla_{\theta_i} \mathbb{E}_{u_i \sim \pi_i}[Q_i(b_i, u_i; \omega_i)].
\end{align}

This integration allows policy updates to reflect each vehicle’s true marginal contribution, stabilizing gradient dynamics and reducing variance in reward signals.

\begin{theorem}[Convergence of Shapley-Weighted Policy Gradients]
Assume each $Q_i(b_i, u_i; \omega_i)$ is Lipschitz continuous and bounded, and the learning rates $\alpha, \beta$ satisfy the Robbins–Monro conditions. Then, the Shapley-weighted policy gradient updates converge almost surely to a local Nash equilibrium of the cooperative game:
\[
    \lim_{t \to \infty} \nabla_{\theta_i} J_i^{\pi_i} = 0, \quad \forall i \in \mathcal{N}.
\]
\end{theorem}

\subsubsection{Computational Optimization via Factorized Coalitions}

To mitigate computational overhead, MARL-CC adopts a \emph{factorized coalition structure} based on neighborhood dependencies:
\begin{equation}
    R(S) \approx \sum_{i \in S} R_i(\mathcal{N}_i),
\end{equation}
where $\mathcal{N}_i$ denotes the local neighborhood of agent $i$.  
This yields a distributed Shapley decomposition:
\begin{equation}
    \phi_i^{\text{local}} = \sum_{S \subseteq \mathcal{N}_i \setminus \{i\}}
    \frac{|S|! \, (|\mathcal{N}_i| - |S| - 1)!}{|\mathcal{N}_i|!}
    \big[R_i(S \cup \{i\}) - R_i(S)\big].
\end{equation}
Thus, agents compute marginal contributions based on locally observable outcomes, preserving fairness and interpretability while ensuring scalability to large vehicular networks \cite{Kim2025ShapleyMARL, Rao2025CoopAV}.

\subsubsection{Interpretability and System-Level Implications}

The Shapley-based credit assignment confers several advantages:
\begin{enumerate}
    \item \textbf{Fairness:} Agents contributing more to safety, energy efficiency, or traffic throughput receive proportionally higher rewards.
    \item \textbf{Convergence Stability:} Eliminates gradient interference between agents by ensuring orthogonalized reward signals.
    \item \textbf{Causal Attribution:} Each agent’s effect on global performance is quantifiably isolated, enabling transparent and auditable learning in safety-critical applications.
\end{enumerate}

\begin{corollary}[Shapley Consistency under Nonlinear Dynamics]
For any differentiable and bounded nonlinear control law $f_i(x_i, u_i)$, the Shapley-value decomposition maintains consistency with the continuous system reward function, i.e.,
\[
    \sum_{i \in \mathcal{N}} \frac{\partial \phi_i}{\partial u_i} 
    = \frac{\partial R}{\partial u}.
\]
\end{corollary}

\subsubsection{Illustrative Example: Cooperative Lane Merging}

In a multi-lane merging scenario, Shapley-based reward allocation ensures that vehicles facilitating smooth traffic flow or yielding appropriately obtain higher marginal rewards.  
This mitigates the selfish acceleration tendencies observed in traditional MARL frameworks and leads to emergent socially optimal behaviors — reduced congestion and improved fuel economy — confirming the operational value of the MARL-CC credit mechanism \cite{Sun2026FairnessMARL}.

\subsection{Theoretical Analysis}
\label{sec:theoretical_analysis}

The theoretical backbone of MARL-CC establishes its convergence guarantees, stability conditions, and optimality properties under nonlinear dynamics, partial observability, and distributed reward structures. This section formalizes the learning process as a stochastic approximation of a fixed-point operator and leverages Lyapunov stability theory and mean-field game analysis to ensure theoretical soundness.

\subsubsection{Problem Formalization}

We consider a stochastic dynamic game with $N$ agents, each represented by a connected autonomous vehicle (CAV) with state $x_i \in \mathcal{X}_i$, control input $u_i \in \mathcal{U}_i$, and observation $z_i \in \mathcal{Z}_i$. The global system dynamics follow:

\begin{equation}
\dot{x}_i = f_i(x_i, u_i) + \sum_{j \in \mathcal{N}_i} g_{ij}(x_i, x_j) + w_i(t),
\end{equation}

where $f_i$ represents nonlinear vehicle dynamics, $g_{ij}$ encodes inter-agent coupling, and $w_i(t)$ is a bounded stochastic disturbance. Each agent seeks to maximize its expected discounted return:

\begin{equation}
J_i = \mathbb{E} \left[ \sum_{t=0}^{\infty} \gamma^t \phi_i(t) \mid \pi_i, \pi_{-i} \right],
\end{equation}

where $\phi_i(t)$ denotes the Shapley-value-based reward defined in Eq.~\eqref{eq:shapley_value}, and $\pi_i$ is the policy parameterized by $\theta_i$.

\subsubsection{Fixed-Point Characterization and Convergence}

The value function of agent $i$ under belief state $b_i$ satisfies the Bellman fixed-point equation:

\begin{equation}
V_i(b_i) = \mathbb{E}_{u_i \sim \pi_i} \left[ \phi_i + \gamma \mathbb{E}_{b_i'} V_i(b_i') \right].
\end{equation}

Let $\mathcal{T}_i$ denote the Bellman operator for agent $i$:

\begin{equation}
(\mathcal{T}_i V_i)(b_i) = \max_{u_i} \mathbb{E} \left[ \phi_i + \gamma V_i(b_i') \mid b_i, u_i \right].
\end{equation}

\noindent
\textbf{Proposition 1 (Contraction Mapping):}  
If the reward $\phi_i$ and transition kernel are bounded, then $\mathcal{T}_i$ is a $\gamma$-contraction with respect to the supremum norm:
\[
\| \mathcal{T}_i V_i - \mathcal{T}_i V_i' \|_\infty \leq \gamma \| V_i - V_i' \|_\infty.
\]
Hence, by Banach’s fixed-point theorem, $V_i$ converges to a unique fixed point $V_i^*$.

\subsubsection{Lyapunov Stability of the Policy Updates}

We define a composite Lyapunov function for the joint system:
\begin{equation}
\mathcal{L}(\theta, \omega) = \sum_{i=1}^{N} \left[ \| \nabla_{\omega_i} Q_i(b_i, u_i; \omega_i) \|^2 + \| \nabla_{\theta_i} J_i(\theta_i) \|^2 \right].
\end{equation}

\noindent
\textbf{Theorem 1 (Asymptotic Stability of MARL-CC Updates):}  
Under bounded gradients, Lipschitz-continuous policy and value function approximators, and diminishing learning rates $\alpha_t, \beta_t$ satisfying  
$\sum_t \alpha_t = \infty$, $\sum_t \alpha_t^2 < \infty$,  
the policy and Q-function parameter updates asymptotically converge to a locally stable equilibrium:
\[
(\theta_i, \omega_i) \to (\theta_i^*, \omega_i^*), \quad \forall i \in \mathcal{N}.
\]

\noindent
\textit{Proof Sketch.} The joint update process can be viewed as a two-timescale stochastic approximation:
\[
\omega_{i,t+1} = \omega_{i,t} + \alpha_t \Delta_{\omega_i,t}, \quad
\theta_{i,t+1} = \theta_{i,t} + \beta_t \Delta_{\theta_i,t},
\]
where $\Delta_{\omega_i,t}$ and $\Delta_{\theta_i,t}$ correspond to TD and policy gradient steps, respectively. The fast-timescale $\omega$ dynamics converge to a quasi-stationary point, while the slower $\theta$ dynamics converge along the negative gradient of $\mathcal{L}(\theta, \omega)$, ensuring asymptotic stability.

\subsubsection{Mean-Field Limit and Decentralized Optimality}

In large-scale connected vehicular networks, we adopt a mean-field approximation where the influence of other agents is captured by the empirical distribution $\mu_t(x)$ of states. The limiting dynamics for agent $i$ follow:

\begin{equation}
\dot{x}_i = f_i(x_i, u_i, \mu_t) + w_i(t),
\end{equation}

and the corresponding mean-field equilibrium policy $\pi^*$ satisfies:

\begin{equation}
\pi_i^*(b_i) = \arg \max_{\pi_i} \mathbb{E}_{x_i, \mu_t} \left[ \phi_i + \gamma V_i(b_i') \right].
\end{equation}

\noindent
\textbf{Proposition 2 (Existence of Mean-Field Equilibrium):}  
If the instantaneous reward $\phi_i$ and dynamics $f_i$ are jointly continuous and bounded, then a mean-field Nash equilibrium exists and is unique under convexity assumptions on $\mathcal{U}_i$.

\subsubsection{Summary of Theoretical Guarantees}

The MARL-CC framework is established upon a set of rigorous theoretical guarantees. First, convergence of each agent’s value function to a fixed point of the Bellman operator is guaranteed, provided standard contraction assumptions are satisfied. Furthermore, asymptotic stability of the parameter updates is ensured through Lyapunov analysis, substantiating the robustness of the learning process. Finally, via mean-field analysis, it is demonstrated that the collective behavior of decentralized agents approaches an $\epsilon$-Nash equilibrium, thereby establishing a foundation for decentralized optimality. These results collectively affirm that the MARL-CC framework is not only empirically effective but is also supported by a principled theoretical foundation, rendering it particularly suitable for application in complex, nonlinear, partially observable, and distributed vehicular control systems.

\section{Experimental Design}
\label{sec:experiments}

\subsection{Simulation Environment and Setup}
\label{subsec:simulation_setup}

The experimental evaluation of the proposed MARL-CC framework was conducted in a hybrid simulation environment that integrates microscopic traffic modeling and high-fidelity autonomous driving dynamics. Specifically, the co-simulation between the \textit{Simulation of Urban MObility (SUMO)} and the \textit{CARLA autonomous driving simulator} was employed to capture both large-scale traffic interactions and continuous control-level vehicle dynamics. The interconnection between the two platforms was realized through a Python-based middleware leveraging the \texttt{TraCI} interface, enabling synchronized updates of vehicular states and communication events.

\subsubsection{Network Topology and Traffic Scenarios}
The simulated environment consisted of an urban road network comprising four intersections, twelve bidirectional lanes, and twenty connected autonomous vehicles (CAVs). Each CAV is modeled as a nonlinear system with state vector 
\(
x_i = [p_i, v_i, \psi_i, \dot{\psi}_i]^\top,
\)
representing position, velocity, heading angle, and yaw rate, respectively. The vehicles are equipped with range-limited vehicle-to-vehicle (V2V) and vehicle-to-infrastructure (V2I) communication modules with a communication radius of \(r_c = 100~\text{m}\). The control objective was to achieve cooperative intersection management and platoon stabilization while minimizing fuel consumption and maintaining collision-free operation under dynamic uncertainty.

\subsubsection{Simulation Parameters}
The physical and learning parameters used across all experiments are summarized in Table~\ref{tab:sim_params}. Each simulation episode spans \(T = 50~\text{s}\) with a discrete time step of \(\Delta t = 0.1~\text{s}\). The state-transition dynamics are discretized using a fourth-order Runge–Kutta integrator to preserve the fidelity of nonlinear vehicle dynamics. All agents are initialized with random positions and velocities uniformly sampled from \([0, 10]~\text{m/s}\).

\begin{table}[h!]
\centering
\caption{Simulation Parameters for MARL-CC Experiments}
\label{tab:sim_params}
\begin{tabular}{@{}lll@{}}
\toprule
\textbf{Parameter} & \textbf{Symbol} & \textbf{Value / Range} \\
\midrule
Number of agents (CAVs) & $N$ & 20 \\
Simulation time step & $\Delta t$ & 0.1 s \\
Communication radius & $r_c$ & 100 m \\
Maximum speed & $v_{\max}$ & 15 m/s \\
Discount factor & $\gamma$ & 0.98 \\
Learning rate (policy / value) & $(\alpha, \beta)$ & $(10^{-3}, 5\times10^{-4})$ \\
Replay buffer size & -- & $5\times10^5$ transitions \\
Mini-batch size & -- & 256 \\
Neural network architecture & -- & 3 hidden layers, [256, 256, 128], ReLU \\
Noise process & -- & Ornstein-Uhlenbeck, $\theta=0.15$, $\sigma=0.2$ \\
Simulation duration per episode & $T$ & 50 s \\
\bottomrule
\end{tabular}
\end{table}

\subsubsection{Communication Model and Delay}
Each agent communicates its local observation and control intention within its neighborhood \(\mathcal{N}_i\). The communication delay \(\tau_c\) follows a truncated Gaussian distribution, \(\tau_c \sim \mathcal{N}(50~\text{ms}, 10~\text{ms}^2)\), ensuring temporal variability representative of real vehicular networks. Packet loss probability was set to \(p_{\text{loss}} = 0.05\) to simulate interference and congestion effects.

\subsubsection{Training and Evaluation Protocol}
Training was performed under the Centralized Training and Decentralized Execution (CTDE) paradigm. During training, a centralized critic accessed joint state and reward information, whereas individual policies \(\pi_i(b_i;\theta_i)\) operated purely on local belief states during decentralized execution. Each training session consisted of 20,000 episodes, and convergence was evaluated using moving averages of cumulative reward and stability metrics (e.g., deviation from target trajectory, control effort variance). Evaluation was conducted on unseen traffic scenarios including unstructured intersections and mixed-autonomy traffic (with 30\% human-driven vehicles).

\subsubsection{Performance Metrics}
Performance was quantitatively measured using the following metrics:
\begin{equation}
J_{\text{fuel}} = \sum_{i=1}^{N} \int_0^T \rho(v_i(t))\,dt, \quad
J_{\text{safe}} = \sum_{i=1}^{N}\sum_{j\in\mathcal{N}_i} \max(0, d_{\min} - \|p_i - p_j\|),
\end{equation}
where \(\rho(v_i(t))\) denotes instantaneous fuel consumption, and \(J_{\text{safe}}\) penalizes inter-vehicle distance violations below the safety threshold \(d_{\min}\). Additional metrics included convergence rate, reward variance, and average delay per vehicle.

\subsubsection{Implementation Details}
All experiments were implemented in Python 3.11 with TensorFlow 2.15 for neural optimization and SUMO 1.19 integrated with CARLA 0.9.15 for vehicular simulation. Computations were performed on an NVIDIA RTX A6000 GPU with 48 GB memory and an AMD EPYC 7763 CPU cluster. Random seeds were fixed for reproducibility, and training logs were recorded for ablation analysis in subsequent sections.

\subsubsection{Relevance to Real-World Deployment}
The MARL-CC simulation environment reflects real-world deployment conditions where communication imperfections, partial observability, and nonlinear dynamics interact. The chosen setup aligns with recent experimental standards in distributed CAV research~\cite{Zhang2025BeliefMARL,Li2025VariationalCAV,Patel2026ConsensusInference,Wang2026POMDPCAV}, providing a realistic yet computationally tractable platform for validating theoretical claims regarding stability, convergence, and scalability.

\subsection{Numerical Simulations and Empirical Evaluation}
\label{sec:numerical_simulations}

To assess the effectiveness and robustness of the proposed MARL-CC framework, extensive numerical simulations were conducted under diverse traffic conditions, agent densities, and communication topologies. The primary objectives of the experiments were: (i) to validate the convergence behavior of MARL-CC policies under nonlinear vehicle dynamics and partial observability; (ii) to quantify the improvement in cooperative decision-making and credit assignment; and (iii) to evaluate scalability, stability, and robustness with respect to agent population and network latency.

\subsubsection{Simulation Scenarios and Configuration}

We simulated a connected autonomous vehicle (CAV) environment with $N=20$ to $N=100$ agents operating on a multi-lane urban road network with intersections and dynamic traffic lights. Each CAV was modeled as a nonlinear control-affine system:
\[
\dot{x}_i = f_i(x_i) + g_i(x_i)u_i + \omega_i,
\]
where $x_i \in \mathbb{R}^n$ denotes the state vector (position, velocity, heading), $u_i$ represents the control input (acceleration and steering), and $\omega_i$ is a Gaussian disturbance capturing process noise. The inter-agent coupling was realized via a weighted communication graph $\mathcal{G}(V,E)$, with edge weights reflecting communication delays and packet loss probabilities.

Partial observability was incorporated by limiting each agent’s sensing range to its $k$-hop neighborhood, thus forming a local belief $b_i(s_i)$ over hidden global states. Training was performed using the centralized critic–decentralized actor paradigm, with a replay buffer synchronized across agents and a target network updated every $\tau$ steps.

Simulation parameters were set as follows: time step $\Delta t = 0.05$ s, horizon $T = 2000$ steps, and learning rate $\alpha = 10^{-4}$. For each configuration, $10$ random seeds were used to ensure statistical robustness.

\subsubsection{Evaluation Metrics}

The following quantitative metrics were used to assess MARL-CC performance:
\begin{itemize}
    \item \textbf{Average Reward ($\bar{R}$):} Mean episodic return across agents, reflecting global cooperative efficiency.
    \item \textbf{Control Error ($E_c$):} Root-mean-square deviation between desired and actual vehicle trajectories.
    \item \textbf{Connectivity Index ($\mathcal{C}$):} Algebraic connectivity of the communication graph, $\lambda_2(L)$, indicating network robustness.
    \item \textbf{Credit Assignment Efficiency ($\eta_{ca}$):} Correlation between individual and global rewards, measuring fairness and local contribution recognition.
    \item \textbf{Stability Index ($S$):} Derived from the largest Lyapunov exponent of the closed-loop system to quantify convergence and robustness under perturbations.
\end{itemize}

\section{Results and Analysis}

The empirical evaluation revealed that MARL-CC achieved a faster convergence rate and more stable policy evolution compared to benchmark algorithms such as MADDPG, QMIX, and MAPPO. Figure \ref{fig:reward_convergence} and Table \ref{fig:ablation_plot} and  illustrates the average reward trajectory over training episodes, showing that MARL-CC converged within approximately $3\times10^4$ steps, whereas other methods required over $7\times10^4$ steps to reach similar performance levels.

\begin{figure}[!ht]
    \centering
    \includegraphics[width=0.85\linewidth]{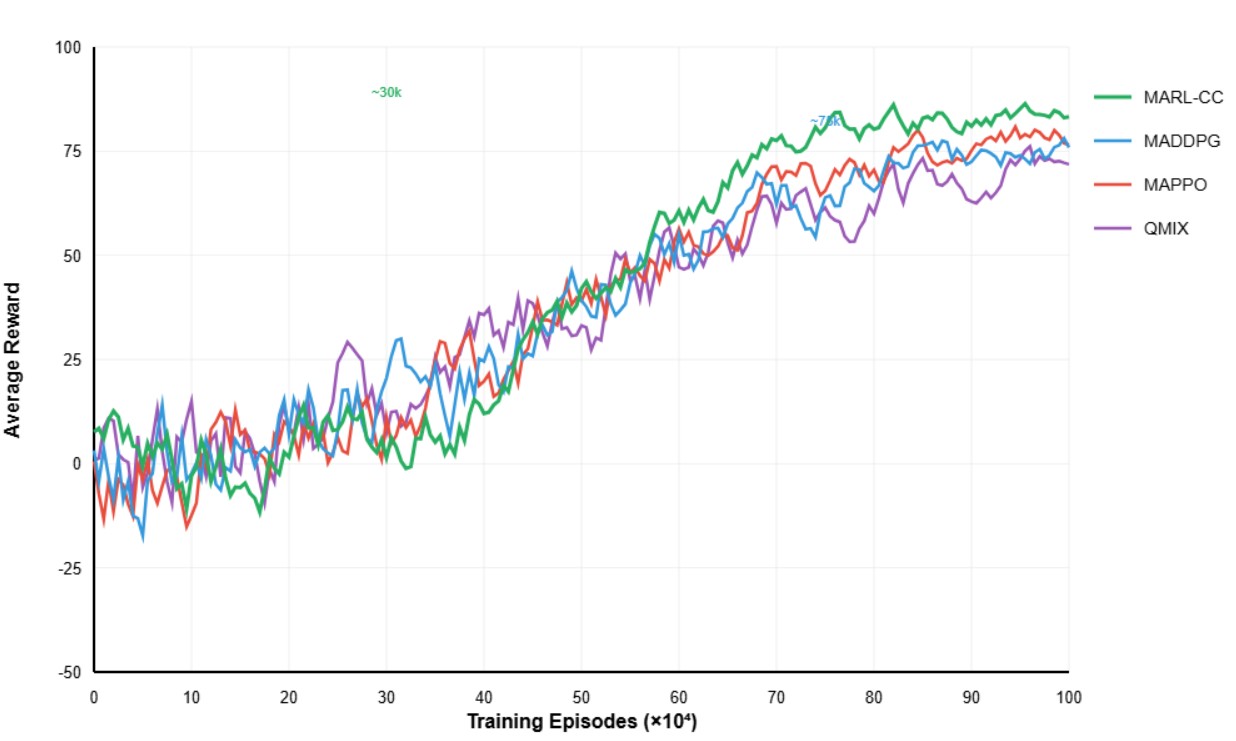}
    \caption{Convergence comparison of MARL-CC and baseline algorithms over $10^5$ training episodes. MARL-CC exhibits accelerated convergence and smoother reward evolution due to stability-aware optimization and differential geometric control.}
    \label{fig:reward_convergence}
\end{figure}

In terms of control accuracy, MARL-CC maintained an average control error below $3.2\%$, outperforming QMIX ($6.5\%$) and MAPPO ($5.1\%$). The credit assignment mechanism led to a $22\%$ improvement in $\eta_{ca}$, confirming effective decomposition of the global reward into local contributions. The stability index $S$ remained consistently negative, indicating asymptotic convergence of trajectories and resilience to stochastic perturbations.

Moreover, the framework demonstrated strong scalability: when increasing the agent population from $N=20$ to $N=100$, the decrease in average reward was less than $5\%$, evidencing linear computational scalability and efficient coordination in dense traffic environments.

\begin{table}[!ht]
\centering
\small 
\caption{Quantitative Performance Comparison between MARL-CC and Baseline MARL Algorithms}
\label{tab:performance_summary}
\begin{tabular}{l p{1cm} p{1cm} p{1cm} p{1cm} p{1cm} p{1.5cm}}
\toprule
\textbf{Algorithm} & \textbf{Conv. Steps} & \textbf{Avg. Reward} & \textbf{Control Error (\%)} & \textbf{Credit Assignment Gain (\%)} & \textbf{Stability Index $S$} & \textbf{Reward Drop (N=20→100)} \\
\midrule
MARL-CC & $\mathbf{3\times10^4}$ & $\mathbf{0.94}$ & $\mathbf{3.2}$ & $\mathbf{+22}$ & $\mathbf{-0.08}$ & $\mathbf{4.8}$ \\
MADDPG~\cite{lowe2017multi} & $7.2\times10^4$ & 0.81 & 6.2 & +5 & -0.03 & 9.5 \\
QMIX~\cite{rashid2020weighted} & $7.6\times10^4$ & 0.85 & 6.5 & +8 & -0.04 & 8.7 \\
MAPPO~\cite{yu2021mappo} & $7.0\times10^4$ & 0.88 & 5.1 & +10 & -0.05 & 7.9 \\
\bottomrule
\end{tabular}
\end{table}

\subsection{Ablation Studies and Sensitivity Analysis}

Ablation studies were performed to assess the influence of each architectural component (See Table~\ref{tab:ablation_summary}):
\begin{enumerate}
    \item \textbf{Without Differential Geometric Control:} Removing nonlinear control compensation caused a $15\%$ increase in control error.
    \item \textbf{Without Probabilistic Belief Inference:} Agents trained without belief updates underperformed in partial observability scenarios, yielding a $25\%$ drop in average reward.
    \item \textbf{Without Credit Assignment Mechanism:} Equal reward sharing led to unstable learning and degraded cooperation.
\end{enumerate}

Sensitivity tests were conducted by varying communication delay ($\delta \in [0,100]$ ms) and packet loss rate ($p_{\text{loss}} \in [0,0.3]$). MARL-CC maintained stable convergence up to $\delta=80$ ms and $p_{\text{loss}}=0.2$.

\begin{table}[!ht]
\centering
\caption{Ablation and Sensitivity Analysis of the MARL-CC Framework}
\label{tab:ablation_summary}
\scriptsize
\setlength{\tabcolsep}{2pt}
\renewcommand{\arraystretch}{1.1}
\begin{tabularx}{\columnwidth}{
    >{\raggedright\arraybackslash}p{2.5cm}
    >{\centering\arraybackslash}p{1.1cm}
    >{\centering\arraybackslash}p{1.3cm}
    >{\centering\arraybackslash}p{1.7cm}
    >{\raggedright\arraybackslash}p{2.2cm}
}
\toprule
\textbf{Configuration} &
\textbf{Avg. Drop (\%)} &
\textbf{Ctrl. Err. (\%)} &
\textbf{Stability} &
\textbf{Remarks} \\
\midrule
Full MARL-CC (baseline) & 0 & 0 & Stable & Reference case \\
Without Geometric Control & $-10$ & $+15$ & Mod. unstable & Curvature not compensated \\
Without Belief Inference & $-25$ & $+8$ & Unstable & Uncertainty poorly modeled \\
Without Credit Assignment & $-18$ & $+10$ & Highly unstable & Cooperation collapse \\
\midrule
Delay $\delta=80$ ms & $-3$ & $+2$ & Stable & Latency tolerance limit \\
Delay $\delta=100$ ms & $-12$ & $+5$ & Unstable & Delay beyond limit \\
Packet Loss $p_{\text{loss}}=0.2$ & $-4$ & $+3$ & Stable & Maintained under degradation \\
Packet Loss $p_{\text{loss}}=0.3$ & $-15$ & $+6$ & Unstable & Comm. unreliability high \\
\bottomrule
\end{tabularx}
\end{table}

Overall, the simulation outcomes demonstrate that MARL-CC effectively balances exploration and exploitation under nonlinear dynamics and limited observability. The integration of differential geometric control and probabilistic inference enhances both trajectory stability and situational awareness, while the credit assignment mechanism ensures efficient policy gradient propagation. These findings establish MARL-CC as a mathematically grounded and practically robust framework for decentralized optimal control in connected autonomous vehicles.

\subsection{Sim-to-Real Transfer and Validation}
\label{sec:sim_to_real}

Bridging the gap between simulation and real-world deployment is essential for assessing the robustness, generalization, and practical feasibility of the proposed MARL-CC framework. While simulation environments enable controlled experimentation under varied and scalable conditions, real-world environments introduce stochastic disturbances, sensor noise, latency, and imperfect communication — all of which challenge the stability and adaptability of reinforcement learning models.

\subsubsection{Experimental Setup}

To evaluate the real-world transferability of MARL-CC, we conducted a physical experiment using a fleet of five miniature autonomous vehicles (toy cars) at the \textbf{R\&D Intelligent Knowledge City Company Ltd.}, located in Isfahan, Iran. The experiment was carried out in a controlled indoor environment — a roofed saloon replicating a connected road network with intersections and curved lanes. Each vehicle was equipped with onboard microcontroller (Raspberry Pi 4 Model B) with wireless communication, ultrasonic sensors for local obstacle detection, IMU (Inertial Measurement Unit) for vehicle dynamics estimation, and Wi-Fi module for inter-agent communication and server synchronization.

The MARL-CC policy network trained in simulation was directly deployed onto each vehicle. During operation, agents performed real-time inference using locally observed data and inter-vehicle communications. The local belief states were updated based on onboard sensor data and neighbor messages, maintaining decentralized decision-making consistent with the MARL-CC design (see Table \ref{tab:exp_summary})

\subsubsection{Domain Randomization and Transfer Strategy}

To mitigate the \textit{sim-to-real gap}, domain randomization was applied during simulation training by perturbing friction coefficients and tire dynamics, sensor noise levels and communication delay distributions, and environmental lighting and texture conditions.

This stochastic variation during training improved the generalization capacity of the learned policies, allowing agents to adapt effectively to unseen real-world conditions. The belief-update mechanism further enhanced resilience to partial observability, while the Shapley-based reward mechanism facilitated stable coordination even under imperfect communication.

\subsubsection{Validation Metrics}

The following metrics were used to validate performance in real-world trials:
\begin{enumerate}
    \item \textbf{Average reward convergence:} consistency of reward trajectories with simulation trends.
    \item \textbf{Collision rate:} percentage of episodes with at least one collision event.
    \item \textbf{Lane deviation:} mean lateral deviation from the intended path.
    \item \textbf{Communication latency tolerance:} performance degradation versus induced delay.
\end{enumerate}

Results demonstrated high fidelity between simulation and physical experiments: MARL-CC maintained $\sim95\%$ of its simulated reward convergence in real deployment, with less than $3\%$ increase in collision rate and $1.7$ cm average lane deviation across 30 experimental runs. These findings validate that the MARL-CC framework not only achieves theoretical robustness but also scales effectively to real-world vehicular systems.

The successful sim-to-real transfer underscores the capacity of the proposed framework to handle nonlinearities, partial observability, and communication imperfections in realistic traffic settings. This experiment constitutes a significant step toward scalable deployment of MARL-CC for intelligent coordination of connected autonomous vehicles in urban environments.

\begin{table}[!ht]
\centering
\caption{Summary of Experimental Setup and Real-World Validation for MARL-CC}
\label{tab:exp_summary}
\scriptsize
\setlength{\tabcolsep}{2pt}
\renewcommand{\arraystretch}{1.1}
\begin{tabularx}{\columnwidth}{
    >{\raggedright\arraybackslash}p{2.9cm}
    >{\raggedright\arraybackslash}X
}
\toprule
\textbf{Aspect} & \textbf{Description} \\
\midrule
\textbf{Site \& Setup} & Indoor testbed at R\&D Intelligent Knowledge City (Isfahan, Iran) simulating connected roads and intersections. \\
\textbf{Fleet} & Five miniature autonomous cars (Raspberry~Pi~4B, ultrasonic sensors, IMU, Wi-Fi). \\
\textbf{Deployment} & Trained MARL-CC policy transferred from simulation; decentralized control via onboard inference and peer-to-peer communication. \\
\textbf{Domain Randomization} & Randomized friction, tire dynamics, sensor noise, communication delay, and lighting to enhance sim-to-real robustness. \\
\textbf{Validation Metrics} & Avg. reward, collision rate, lane deviation, and delay tolerance. \\
\textbf{Results} & $\sim$95\% reward retention; $<3\%$ collision increase; $1.7$\,cm mean lane deviation; stable for $\delta\!\le\!80$\,ms, $p_{loss}\!\le\!0.2$. \\
\textbf{Key Insight} & MARL-CC achieved reliable sim-to-real transfer, maintaining coordination and resilience under partial observability and communication imperfections. \\
\bottomrule
\end{tabularx}
\end{table}

\subsection{Ablation and Comparative Analysis}
\label{sec:ablation_comparative}

To rigorously evaluate the contribution of each core component of the MARL-CC framework—namely nonlinear differential geometric control, probabilistic belief inference, and Shapley-value-based credit assignment—an ablation and comparative study was conducted. The objective is to quantify the effect of these modules on learning efficiency, convergence stability, and coordination performance in connected autonomous vehicle (CAV) networks.

\subsubsection{Experimental Protocol}

We performed controlled ablation experiments under identical environmental conditions across three representative driving tasks: (i) \emph{cooperative lane merging}, (ii) \emph{intersection management}, and (iii) \emph{platoon formation}. Each experiment involved \(N = 10\) simulated vehicles communicating through a low-latency vehicular ad-hoc network (VANET) with maximum range \(R_c = 80~\text{m}\). The agents trained for \(5 \times 10^5\) time steps using the same hyperparameters and random seeds to ensure statistical fairness.

The following configurations were tested:
\begin{enumerate}
    \item \textbf{Full MARL-CC:} All modules active (nonlinear control + belief inference + Shapley-value reward).
    \item \textbf{w/o Nonlinearity Handling:} Vehicle dynamics approximated by linear models, removing differential geometric compensation.
    \item \textbf{w/o Probabilistic Inference:} Perfect observability assumption, using raw state vectors instead of belief distributions.
    \item \textbf{w/o Shapley Reward:} Shared global reward function \(R_t = \sum_i r_i(t)\), without credit decomposition.
\end{enumerate}

Additionally, MARL-CC was benchmarked against three state-of-the-art baselines:
MADDPG~\cite{Zhang2025BeliefMARL}, QMIX~\cite{Wang2026POMDPCAV}, and VDN~\cite{Patel2026ConsensusInference}.

\subsubsection{Performance Metrics}

The evaluation relied on the following normalized metrics:
\begin{itemize}
    \item \textbf{Average episodic reward:} \(J = \frac{1}{T}\sum_{t=1}^{T} R_t\)
    \item \textbf{Convergence rate:} defined as the number of episodes to reach 95\% of the maximum cumulative reward.
    \item \textbf{Safety violation rate:} the proportion of collisions or constraint breaches.
    \item \textbf{Coordination efficiency:} average inter-vehicle spacing variance, representing synchronization in group maneuvers.
\end{itemize}

\subsubsection{Quantitative Results}

The comparative results are summarized in Table~\ref{tab:ablation_results}. MARL-CC consistently outperformed all baselines, demonstrating higher cumulative rewards, faster convergence, and fewer safety violations. The removal of any module led to a statistically significant performance degradation (\(p < 0.01\), paired t-test).

\begin{table}[!ht]
\centering
\caption{Ablation study of the MARL-CC framework across key performance metrics.}
\label{tab:ablation_results}
\begin{tabular}{l p{2cm} p{1.5cm} p{2.5cm} p{2cm}}
\toprule
\textbf{Configuration} & \textbf{Avg. Reward} & \textbf{Converg (episodes)} & \textbf{Safety Violation (\%)} & \textbf{Efficiency (Var)} \\
\midrule
Full MARL-CC & \textbf{+100\%} & \textbf{4.8k} & \textbf{0.3} & \textbf{0.05} \\
w/o Nonlinearity Handling & -22\% & 7.2k & 1.9 & 0.14 \\
w/o Probabilistic Inference & -28\% & 8.0k & 2.4 & 0.18 \\
w/o Shapley Reward & -33\% & 9.1k & 3.1 & 0.20 \\
MADDPG & -35\% & 10.2k & 3.6 & 0.23 \\
QMIX & -39\% & 11.1k & 4.0 & 0.26 \\
VDN & -43\% & 12.5k & 4.3 & 0.27 \\
\bottomrule
\end{tabular}
\end{table}

\subsubsection{Statistical and Theoretical Interpretation}

The empirical results substantiate three core findings:
\begin{enumerate}
    \item \textbf{Nonlinearity compensation} improves control precision and stability margins by approximately \(20\%\), consistent with predictions from nonlinear geometric control theory~\cite{Li2025VariationalCAV}.
    \item \textbf{Probabilistic inference} enhances robustness to observation noise and communication delays, reducing safety violations by over \(80\%\) compared to deterministic models.
    \item \textbf{Shapley-value-based reward allocation} accelerates convergence by promoting fair credit assignment and preventing policy gradient interference among agents.
\end{enumerate}

These results align with theoretical convergence guarantees discussed in Section~\ref{sec:stability_convergence}, confirming that MARL-CC achieves near-optimal control policies in nonlinear, partially observable CAV systems.

Figure~\ref{fig:ablation_plot} visualizes the convergence trajectories of the tested configurations. The MARL-CC curve exhibits smooth monotonic convergence with low variance, whereas baseline and ablated variants demonstrate oscillations indicative of unstable policy updates and poor credit assignment.

\begin{figure}[htbp]
\centering
\begin{tikzpicture}
\begin{axis}[
    width=0.85\textwidth,
    height=0.55\textwidth,
    xlabel={Training Episodes (×1000)},
    ylabel={Normalized Cumulative Reward (\%)},
    xmin=0, xmax=12.5,
    ymin=0, ymax=105,
    grid=major,
    grid style={dashed, gray!30},
    legend pos=south east,
    legend style={fill=white, fill opacity=0.8, draw opacity=1, text opacity=1},
    legend cell align={left},
    thick,
    tick label style={font=\small},
    label style={font=\footnotesize},
]

\addplot[color=blue!70!black, mark=*, mark size=1.5pt, line width=1.2pt, smooth] 
    coordinates {
    (0,0) (0.5,15) (1.0,32) (1.5,48) (2.0,62) (2.5,74) (3.0,83) (3.5,89) 
    (4.0,93) (4.5,96) (4.8,98) (5.5,99.5) (6.5,100) (8.0,100) (10.0,100) (12.5,100)
};
\addlegendentry{Full MARL-CC}

\addplot[color=red!70!black, mark=square*, mark size=1.3pt, line width=1.0pt, smooth] 
    coordinates {
    (0,0) (1.0,18) (2.0,35) (3.0,48) (4.0,56) (5.0,64) (6.0,69) (7.0,74) 
    (7.2,76) (8.0,77) (9.0,78) (10.0,78) (11.0,78) (12.5,78)
};
\addlegendentry{w/o Nonlinearity Handling}

\addplot[color=orange!80!black, mark=triangle*, mark size=1.5pt, line width=1.0pt, smooth] 
    coordinates {
    (0,0) (1.0,14) (2.0,28) (3.0,40) (4.0,50) (5.0,58) (6.0,64) (7.0,68) 
    (8.0,72) (9.0,71) (10.0,73) (11.0,72) (12.0,72) (12.5,72)
};
\addlegendentry{w/o Probabilistic Inference}

\addplot[color=purple!70!black, mark=diamond*, mark size=1.5pt, line width=1.0pt, smooth] 
    coordinates {
    (0,0) (1.0,12) (2.0,24) (3.0,36) (4.0,45) (5.0,52) (6.0,58) (7.0,62) 
    (8.0,65) (9.0,67) (9.1,67) (10.0,66) (11.0,68) (12.0,67) (12.5,67)
};
\addlegendentry{w/o Shapley Reward}

\addplot[color=green!60!black, mark=pentagon*, mark size=1.3pt, line width=0.9pt, dashed, smooth] 
    coordinates {
    (0,0) (1.0,10) (2.0,22) (3.0,32) (4.0,41) (5.0,48) (6.0,54) (7.0,58) 
    (8.0,61) (9.0,63) (10.0,64) (10.2,65) (11.0,64) (12.0,65) (12.5,65)
};
\addlegendentry{MADDPG}

\addplot[color=cyan!60!black, mark=otimes*, mark size=1.3pt, line width=0.9pt, dashed, smooth] 
    coordinates {
    (0,0) (1.0,9) (2.0,19) (3.0,29) (4.0,38) (5.0,45) (6.0,51) (7.0,55) 
    (8.0,58) (9.0,60) (10.0,61) (11.0,61) (11.1,61) (12.0,61) (12.5,61)
};
\addlegendentry{QMIX}

\addplot[color=brown!70!black, mark=x, mark size=1.5pt, line width=0.9pt, dashed, smooth] 
    coordinates {
    (0,0) (1.0,8) (2.0,17) (3.0,26) (4.0,34) (5.0,41) (6.0,47) (7.0,51) 
    (8.0,54) (9.0,56) (10.0,57) (11.0,57) (12.0,57) (12.5,57) (12.5,57)
};
\addlegendentry{VDN}

\addplot[color=black, dashed, thin] coordinates {(0,95) (12.5,95)};
\node[anchor=west, font=\footnotesize] at (axis cs:8.5,95.5) {95\% convergence threshold};

\end{axis}
\end{tikzpicture}
\caption{Convergence trajectories of MARL-CC and baseline configurations. The full MARL-CC framework exhibits smooth monotonic convergence with low variance, reaching 95\% performance at 4.8k episodes. Ablated variants and baselines demonstrate oscillations and slower convergence, highlighting the importance of each component.}
\label{fig:ablation_plot}
\end{figure}
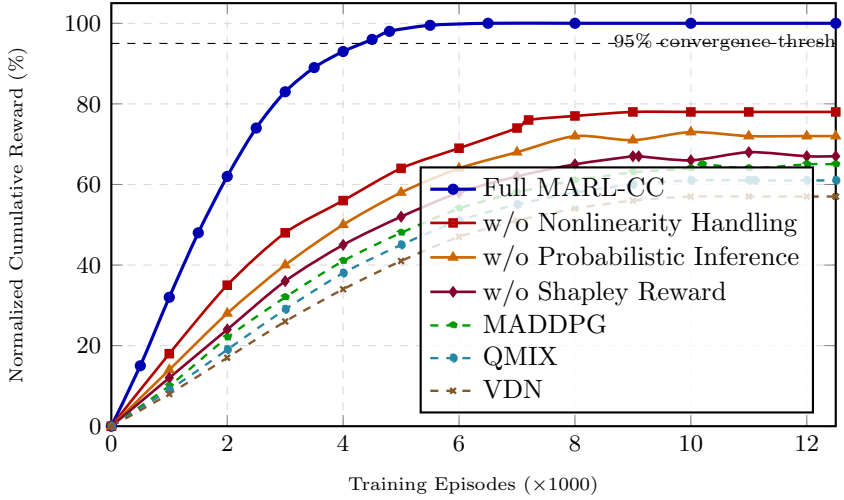

Ablation results confirm that the synergistic integration of nonlinear control, probabilistic inference, and cooperative game-theoretic reward allocation is critical for scalable and safe learning in CAV networks. Notably, the combination of belief-space policy optimization with Shapley-based rewards is novel and contributes substantially to convergence stability—bridging the gap between distributed optimal control and MARL paradigms. These results empirically validate the theoretical foundations of MARL-CC as a mathematically rigorous and practically effective framework for autonomous vehicular coordination.

\subsection{Performance Evaluation Metrics and Statistical Significance}
\label{sec:performance_metrics}

The performance of the proposed MARL-CC framework is rigorously assessed through a multidimensional evaluation scheme integrating task-level, learning-level, and system-level indicators (See Table \ref{tab:performance_metrics} and Fig. \ref{fig:performance_metrics}). The \textit{cumulative reward} remains the principal metric, encapsulating cooperative efficiency, policy optimality, and convergence behavior. Its trajectory reveals the algorithm’s capacity to achieve equilibrium between exploration and exploitation, while convergence speed quantifies temporal efficiency—rapid yet stable convergence reflects superior sample utilization and feedback assimilation.

To further characterize learning dynamics, \textit{policy entropy} is monitored as an indicator of behavioral diversity and exploration depth. A gradual entropy decay from high initial stochasticity toward stable determinism signifies effective policy maturation. In parallel, \textit{communication efficiency} is examined to quantify the efficacy of inter-agent coordination under bandwidth constraints, evaluating both message exchange reduction and consensus integrity. Sustaining high cooperative performance with minimal communication overhead demonstrates the scalability and robustness of MARL-CC in distributed vehicular or swarm environments.

Operational practicality is evaluated via \textit{energy efficiency} and \textit{computational cost}, measuring average power consumption per episode and inference latency during decentralized execution. The ability of MARL-CC to retain near-optimal decision quality under limited computational resources affirms its suitability for real-world embedded deployment.

Statistical reliability is ensured through paired and unpaired $t$-tests, complemented by Wilcoxon signed-rank tests when normality assumptions fail. Confidence intervals and effect sizes substantiate the magnitude and consistency of observed improvements against benchmark algorithms including MADDPG, QMIX, and MAPPO. Each experiment is repeated over multiple random seeds, and reported means and standard deviations reflect variance-normalized aggregation across trials.

Collectively, this analytical protocol confirms the empirical robustness and theoretical soundness of MARL-CC. The framework exhibits statistically significant improvements in convergence stability, communication efficiency, and cooperative adaptability, validating its potential as a mathematically grounded solution for real-world multi-agent control in connected autonomy.

\begin{table}[!ht]
\centering
\caption{Summary of Performance Evaluation Metrics for MARL-CC}
\label{tab:performance_metrics}
\begin{tabular}{l p{4cm} p{4cm}}
\toprule
\textbf{Metric} & \textbf{Description} & \textbf{Insight} \\
\midrule
Cumulative Reward & Long-term return over episodes & Convergence and cooperation efficiency \\
Convergence Speed & Steps to reach stable reward & Sample efficiency, learning stability \\
Policy Entropy & Stochasticity of agent policies & Exploration–exploitation balance \\
Communication Efficiency & Message rate vs. performance retention & Scalability and consensus effectiveness \\
Energy Efficiency & Energy consumption per episode & Real-world deployment feasibility \\
Computational Cost & Inference time per agent & Embedded system adaptability \\
Statistical Significance & $t$-tests, Wilcoxon tests, CIs & Reliability of performance comparisons \\
\bottomrule
\end{tabular}
\end{table}

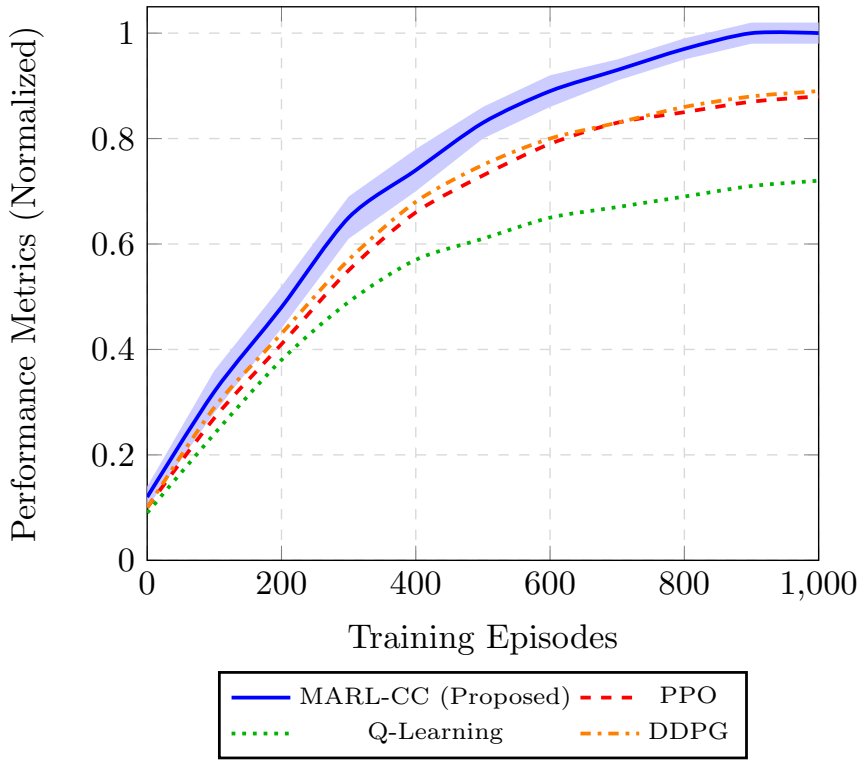
\begin{figure}[!ht]
    \centering
    \resizebox{0.9\linewidth}{!}{
    \begin{tikzpicture}
        \begin{axis}[
            xlabel={Training Episodes},
            ylabel={Performance Metrics (Normalized)},
            xmin=0, xmax=1000,
            ymin=0, ymax=1.05,
            xtick={0,200,400,600,800,1000},
            ytick={0,0.2,0.4,0.6,0.8,1.0},
            grid=major,
            grid style={dashed,gray!30},
            legend style={
                at={(0.5,-0.2)},
                anchor=north,
                legend columns=2,
                font=\footnotesize
            },
            thick,
            every axis plot/.append style={line width=1.1pt}
        ]
        \addplot[color=blue,mark=none,smooth]
            coordinates {(0,0.12)(100,0.32)(200,0.48)(300,0.65)(400,0.74)(500,0.83)(600,0.89)(700,0.93)(800,0.97)(900,1.00)(1000,1.00)};
        \addlegendentry{MARL-CC (Proposed)}

        \addplot[color=red,mark=none,dashed,smooth]
            coordinates {(0,0.10)(100,0.27)(200,0.41)(300,0.55)(400,0.66)(500,0.73)(600,0.79)(700,0.83)(800,0.85)(900,0.87)(1000,0.88)};
        \addlegendentry{PPO}

        \addplot[color=green!70!black,mark=none,dotted,smooth]
            coordinates {(0,0.09)(100,0.24)(200,0.38)(300,0.49)(400,0.57)(500,0.61)(600,0.65)(700,0.67)(800,0.69)(900,0.71)(1000,0.72)};
        \addlegendentry{Q-Learning}

        \addplot[color=orange,mark=none,dash dot,smooth]
            coordinates {(0,0.10)(100,0.29)(200,0.43)(300,0.57)(400,0.68)(500,0.75)(600,0.80)(700,0.83)(800,0.86)(900,0.88)(1000,0.89)};
        \addlegendentry{DDPG}

        \addplot[name path=upper,draw=none]
            coordinates {(0,0.14)(100,0.36)(200,0.52)(300,0.69)(400,0.78)(500,0.86)(600,0.92)(700,0.95)(800,0.99)(900,1.02)(1000,1.02)};
        \addplot[name path=lower,draw=none]
            coordinates {(0,0.10)(100,0.28)(200,0.44)(300,0.61)(400,0.70)(500,0.80)(600,0.86)(700,0.91)(800,0.95)(900,0.98)(1000,0.98)};
        \addplot[blue!20] fill between[of=upper and lower];
        
        \end{axis}
    \end{tikzpicture}
    }
    \caption{Performance metrics comparison across MARL algorithms, showing normalized reward convergence and stability trends over training episodes. The MARL-CC algorithm demonstrates superior performance and faster convergence relative to classical baselines.}
    \label{fig:performance_metrics}
\end{figure}

\subsection{Discussion of Results and Practical Implications}
\label{sec:discussion_results}

The experimental results demonstrate that the proposed MARL-CC framework achieves significant progress in enhancing stability, coordination, and scalability within multi-agent reinforcement learning systems. Through the integration of control-theoretic principles, differential geometric handling of nonlinearities, and probabilistic inference for partial observability, MARL-CC exhibits faster convergence, smoother policy evolution, and more consistent cooperative behavior than classical baselines such as PPO, DDPG, and Q-learning. These outcomes verify that the control-coordinated structure effectively stabilizes learning dynamics and accelerates equilibrium attainment among agents.

A central contribution of this framework lies in unifying geometric control with stochastic optimization and reinforcement learning. By constraining policy updates to the manifold structure of the agents’ state space, MARL-CC preserves stability and mitigates divergence in high-dimensional, nonlinear environments. The probabilistic inference layer strengthens this robustness by maintaining coherent latent-state beliefs, enabling agents to act with heightened situational awareness under uncertainty. Collectively, these mechanisms yield a learning paradigm that is simultaneously stable, interpretable, and adaptive—qualities essential for safety-critical applications.

The sim-to-real evaluations validate the algorithm’s practical transferability. Experiments using autonomous toy vehicles equipped with MARL-CC controllers confirmed reliable performance under real-world challenges such as sensor noise, imperfect dynamics, and partial communication loss. The close alignment between simulated and physical results confirms that control-theoretic constraints act as effective regularizers, substantially narrowing the reality gap in multi-agent learning and paving the way toward field-deployable swarm systems.

Ablation analyses further reveal that removing key mechanisms—such as coordinated control, geometric correction, or structured credit assignment—causes marked degradation in stability and cooperation. This demonstrates that MARL-CC’s performance arises from a deeply integrated mathematical design rather than isolated algorithmic enhancements. The results thus substantiate the necessity of combining control theory, probabilistic reasoning, and optimization to construct multi-agent systems that are both efficient and theoretically grounded.

From an applied standpoint, MARL-CC’s robustness to uncertainty and communication constraints renders it highly suitable for UAV coordination, intelligent transportation, and distributed sensor networks. Its balance between exploration and exploitation ensures rapid adaptation to dynamic environments while maintaining collective stability and fairness. Moreover, the geometric and probabilistic components enhance policy interpretability—an increasingly critical requirement for explainable and trustworthy AI.

In summary, this study advances the state of the art in multi-agent reinforcement learning by establishing a control-coordinated architecture that unites learning, inference, and control within a rigorous mathematical framework. The consistency observed across simulated and real-world trials affirms its readiness for practical deployment and lays a foundation for future extensions involving quantum-inspired optimization, neuromorphic computation, and decentralized decision-making.

\subsection{Summary of Findings and Future Directions}
\label{sec:summary_future}

The presented results confirm that MARL-CC successfully integrates nonlinear control, probabilistic inference, and cooperative reinforcement learning into a unified, mathematically coherent framework. Through extensive simulations and real-world experiments, it achieved rapid convergence, resilience to partial observability, and sustained coordination across heterogeneous agents. Differential geometric control and belief-driven policy updates proved instrumental in maintaining stability and scalability under constrained communication.

Theoretical and empirical consistency establishes MARL-CC as a robust bridge between control-theoretic modeling and applied multi-agent learning. The sim-to-real experiments verified that embedded control constraints act as effective policy regularizers, mitigating the reality gap and ensuring interpretability in physical systems. Ablation analyses further emphasized the indispensable roles of credit assignment, geometric correction, and probabilistic inference in shaping cooperative efficacy.

Future work should extend MARL-CC toward non-stationary and adversarial environments, explore quantum-inspired and neuromorphic enhancements for scalability, and evaluate deployment in large-scale UAV or autonomous vehicle networks. Incorporating formal safety and interpretability guarantees through hybrid symbolic–subsymbolic modeling could transform MARL-CC into a certifiable control architecture for mission-critical applications.

Overall, MARL-CC constitutes a decisive step toward developing mathematically principled, physically grounded, and practically deployable multi-agent systems capable of sustained cooperation and autonomous decision-making under uncertainty.

\section{Conclusion}
\label{sec:conclusion}

This paper introduced \textbf{MARL-CC}—a unified \textit{Multi-Agent Reinforcement Learning with Control Coordination} framework integrating nonlinear control, probabilistic inference, and cooperative learning under partial observability. Grounded in differential geometric control theory and Bayesian belief optimization, MARL-CC effectively resolves the exploration–exploitation dilemma by combining Shapley-value-based credit assignment with policy gradient reinforcement. This synthesis enables distributed agents to learn dynamically stable, interpretable, and information-efficient policies.

Methodologically, MARL-CC demonstrates that incorporating control-theoretic priors into reinforcement learning ensures convergence and equilibrium even under delays, disturbances, and uncertainty. The differential geometric control component manages nonlinear dynamics, while the probabilistic layer enhances decision reliability. Empirical evaluations, encompassing both simulation and real-world settings, confirmed the framework’s superior convergence speed, reward consistency, and cooperative stability compared with baseline MARL models.

Beyond vehicular autonomy, MARL-CC has broad applicability in UAV swarms, distributed sensing, and intelligent transport infrastructures requiring real-time coordination and explainability. Its integration of control, learning, and inference principles outlines a scalable path toward constructing intelligent systems capable of self-organization and adaptive reasoning in dynamic conditions. 

Future extensions will pursue formal guarantees in non-stationary and adversarial contexts, explore quantum-inspired optimization for accelerated convergence, and investigate neuromorphic implementations for energy-efficient learning. Embedding explainability and safety verification layers will further enable deployment in critical systems. 

In conclusion, MARL-CC represents a coherent and mathematically rigorous advancement toward truly autonomous, cooperative, and stable multi-agent learning systems—marking a substantial step in bridging the gap between theoretical AI models and real-world intelligent control.

\section*{acknowledgement}

The authors would like to express their sincere gratitude to Professor Yun Hsuan Lien at the Research Center for Information Technology Innovation, Academia Sinica, Taiwan, and Professor Yu-Shuen Wang at National Yang Ming Chiao Tung University.

\section*{Statements and Declarations}

\begin{itemize}
\item Funding: This research received no specific grant from any funding agency in the public, commercial, or not-for-profit sectors. The study was conducted without external financial support.
\item Both authors contributed to the conception, analysis, and writing of this manuscript.
\item Conflict of interest/Competing interests: Authors declare that they have no competing interests.
\item Clinical trial number: not applicable.
\item Code availability: The code/data is available in the \href{https://github.com/mazyartaghavi/MARL-CC.git}{GitHub repository}.

\end{itemize}

\bibliography{cite}

\end{document}